\documentclass[10pt,reqno]{amsart}

\usepackage{hyperref}
\usepackage{xparse}
\NewDocumentCommand{\ceil}{s O{} m}{%
  \IfBooleanTF{#1} 
    {\left\lceil#3\right\rceil} 
    {#2\lceil#3#2\rceil} 
}

\usepackage{enumerate}
\usepackage{mathrsfs}

\usepackage{color}

\newtheorem{theorem}{Theorem}[section]
\newtheorem{lemma}[theorem]{Lemma}

\newtheorem{prop}[theorem]{Proposition}

\theoremstyle{remark}
\newtheorem{remark}[theorem]{\bf{Remark}}

\theoremstyle{definition}
\newtheorem{assumption}[theorem]{Assumption}

\newtheorem{definition}[theorem]{Definition}

\newcommand\cbrk{\text{$]$\kern-.15em$]$}}
\newcommand\opar{\text{\,\raise.2ex\hbox{${\scriptstyle
|}$}\kern-.34em$($}}
\newcommand\cpar{\text{$)$\kern-.34em\raise.2ex\hbox{${\scriptstyle |}$}}\,}
\newcommand\cbar{\text{$[$\kern-.15em$[$}}

\newcommand\bC{\mathbb{C}}
\newcommand\bL{\mathbb{L}}
\newcommand\bR{\mathbb{R}}
\newcommand\bH{\mathbb{H}}
\newcommand\bZ{\mathbb{Z}}

\newcommand\bE{\mathbb{E}}
\newcommand\bN{\mathbb{N}}

\newcommand\cA{\mathcal{A}}
\newcommand\cB{\mathcal{B}}
\newcommand\cD{\mathcal{D}}
\newcommand\cF{\mathcal{F}}

\newcommand\cH{\mathcal{H}}

\newcommand\cP{\mathcal{P}}

\newcommand\cS{\mathcal{S}}
\newcommand\cZ{\mathcal{Z}}

\newcommand{\mysection}[1]{\section{#1}
\setcounter{equation}{0}}

\newcommand{\Ccinf}{C_{c}^{\infty}}
\newcommand{\R}{\mathbb{R}}

\begin{document}

\title[TIME-FRACTIONAL SPDES DRIVEN BY L\'EVY PROCESSES]
{A Sobolev space theory for the time-fractional  stochastic partial differential equations driven by L\'evy processes}

\author{Kyeong-Hun Kim}
\address{Department of Mathematics, Korea University, 1 Anam-dong,
Sungbuk-gu, Seoul, 136-701, Republic of Korea}
\email{kyeonghun@korea.ac.kr}

\thanks{The authors were  supported by the National Research Foundation of Korea(NRF) grant funded by the Korea government(MSIT) (No. NRF-2019R1A5A1028324)}

\author{Daehan Park}
\address{Stochastic Analysis and Application Research Center, Korea Advanced Institute of Science and Technology, 291 Daehak-ro, Yuseong-gu, Daejeon, 34141, Republic of Korea}
\email{daehanpark@kaist.ac.kr}



\begin{abstract}
We present an  $L_{p}$-theory ($p\geq 2$) for semi-linear time-fractional stochastic partial differential equations driven by L\'evy processes of the type
$$
\partial^{\alpha}_{t}u=\sum_{i,j=1}^d a^{ij}u_{x^{i}x^{j}} +f(u)+\sum_{k=1}^{\infty}\partial^{\beta}_{t}\int_{0}^{t} (\sum_{i=1}^d\mu^{ik} u_{x^i} +g^k(u)) dZ^k_{s}
$$
given with nonzero intial data. Here $\partial^{\alpha}_t$ and $\partial^{\beta}_t$ are the Caputo fractional derivatives, 
$$0<\alpha<2, \qquad \beta<\alpha+1/p,
$$ and  $\{Z^k_t:k=1,2,\cdots\}$ is a sequence of independent L\'evy processes. The coefficients are random functions depending on $(t,x)$.   We prove the uniqueness and existence results in Sobolev spaces, and obtain the maximal regularity of the solution.

 As an application, we also  obtain an $L_p$-regularity theory of  the equation
$$
\partial^{\alpha}_{t}u=\sum_{i,j=1}^d a^{ij}u_{x^{i}x^{j}} +f(u)+\partial^{\beta}_t \int^t_0 h(u)d\mathcal{Z}_s,
$$
where $\dot{\mathcal{Z}}_t$ is a multi-dimensional L\'evy space-time  white noise  with the space dimension
$ d<4-\frac{2(2\beta-2/p)^{+}}{\alpha}$. 
In particular, if $\beta<\alpha/4+1/p$ then one can take $d=1,2,3$.
\end{abstract}

\maketitle

\mysection{introduction}
Let $\{W^{k}_{t}:k\in 1,2,\cdots\}$ and $\{Z^{k}_{t}:k=1,2,\cdots\}$ be sequences of independent  one dimensional Brownian motions and $d_1$-dimensional L\'evy processes respectively.   In this  article we  present an $L_p$-theory ($p\geq 2$)  for the time-fractional stochastic partial differential equation  (SPDE) defined on $\bR^d$:
\begin{equation}\label{eqn 07.16.2}
\begin{aligned}
\partial^{\alpha}_{t}u=&a^{ij}u_{x^ix^j}+b^{i}u_{x^i}+cu+f(u)
\\
&+\partial^{\beta_{1}}_{t}\int_{0}^{t}\left(\mu^{ik}u_{x^i}+\nu^{k}u+g^{k}(u)\right)dW^{k}_{s}
\\
&+\partial^{\beta_{2}}_{t}\int_{0}^{t}\left(\bar{\mu}^{irk}u_{x^i}+\bar{\nu}^{rk}u+h^{rk}(u)\right)dZ^{rk}_{s}, \quad t>0,
\\
&u(0,\cdot)=u_{0},\quad 1_{\alpha>1} \partial_{t} u(0,\cdot)=1_{\alpha>1}v_{0}.
\end{aligned}
\end{equation}
Here $\partial^{\alpha}_{t}, \partial^{\beta_1}_t, \partial^{\beta_2}_t$ are  the Caputo fractional derivatives, 
\begin{equation}
 \label{albe}
 \alpha\in(0,2), \quad \beta_{1}<\alpha+1/2, \quad \beta_{2}<\alpha+1/p,
 \end{equation}
  and  Einstein's summation convention is used  in \eqref{eqn 07.16.2} for the repeated indices $i,j\in \{1,2,\cdots,d\}$, $r\in \{1,2,\cdots, d_1\}$ and $k=1,2,\cdots$. The coefficients depend on $(\omega,t,x)$ and initial data depend on $(\omega,x)$. The   conditions on $\beta_1$ and $\beta_2$ in \eqref{albe} are necessary  and  will  be discussed later (see Remark \ref{remark con}).

Equation \eqref{eqn 07.16.2}  is understood by its integral form (see Definition \ref{def 07.11.1}), and this type of
SPDE naturally arises, for instance, when one  describes the random effects on transport of
particles   subject to sticking and trapping or particles in medium with thermal memory. See \cite{chen2015fractional}  for a detailed derivation of such equations.

This article is a natural continuation of   \cite{kim16timefractionalspde}, which deals with the equation driven  by Wiener processes. We extend the result of  \cite{kim16timefractionalspde}   to more general equation, that is, equation \eqref{eqn 07.16.2} driven also by L\'evy processes. Such generalization is important in applications since the  random effects on natural phenomena  can be discontinuous in time. Another generalization is that, unlike in \cite{kim16timefractionalspde}, we impose non-zero initial data.  Actually, even for the  deterministic  initial-value problem
$$
\partial^{\alpha}_tu=\Delta u, \quad t>0\,; \quad u(0,\cdot)=u_0,\,\, 1_{\alpha>1}\partial_tu(0)=1_{\alpha>1}v_0,
$$
 our result is partially new and  an extension of \cite[Theorem 3.1]{zacher2006quasilinear}, which is based on the semi-group approach and requires some extra algebraic conditions such as 
$\alpha\not\in \{\frac{1}{p},  1+\frac{1}{p}\}$. Our approach is based on Littlewood-Paley theory and we only assue $\alpha \in (0,2)$.

To explain a  technical difference between the equation with Wiener processes and the equation with L\'evy processes, let us consider the model equation
$$
 \partial^{\alpha}_t u=\Delta u+\partial^{\beta}_t\int^t_0 h(s) \, dX_t, \,\,\, t>0 \quad ; \quad u(0)=1_{\alpha>1}\partial_{t}u(0)=0.
 $$
 It turns out that if $X_t$ is a Wiener process, then 
 for any $n\geq 0$ and $p\geq 2$, we have 
\begin{equation}
\label{eqn wiener}
\|D^n_x u\|^p_{\bL_p(T)} \leq C  \Big\| \left( \int^t_0  \left| \big(D^n_x D^{\beta-\alpha}_t p(t-s,\cdot)\big) \ast_x h(s,\cdot) \right|^2 ds\right)^{1/2}\Big\|^p_{\bL_p(T)}, 
\end{equation}
where $\bL_p=L_p(\Omega\times [0,T]; L_p(\bR^d))$, and $p(t,x)$ is the fundamental solution to the fractional heat equation $\partial^{\alpha}_tv=\Delta v$. On the other hand, if $X_t$ is a L\'evy process, then we have
\begin{equation}
\label{levy}
\|D^n_xu\|^p_{\bL_p(T)}\leq C  \Big\|\int^t_0 \left| \big(D^n_xD^{\beta-\alpha}_t p(t-s,\cdot)\big) \ast_x h(s,\cdot) \right|^pds \Big\|_{L_1(\Omega\times [0,T]; L_1(\bR^d))}.
\end{equation}
A sharp estimate of the right hand side of \eqref{eqn wiener} is introduced in \cite{kim16timefractionalspde}, and in this article we obtain a sharp upper bound of  the right hand side of \eqref{levy} with the help of  Littlewood-Paley theory in harmonic analysis.

Below we introduce some related results. To the best of our knowledge,  the regularity result for the time fractional SPDE was firstly introduced in  \cite{desch2013maximal, desch2008stochastic, desch2011p}. The authors in  \cite{desch2013maximal, desch2008stochastic, desch2011p}  applied $H^{\infty}$-functional calculus technique to obtain   a maximal
 regularity   for the mild solution to
  the integral equation 
\begin{equation}
               \label{eqn h}
U(t)+ \int^t_0(t-s)^{\alpha-1}AU(s)ds=\int^t_0(t-s)^{\beta-1}G(s)dW_s,
\end{equation}
where $W_t$ is a Brownian motion, and $A$ is the generator of a bounded analytic semigroup   and is assumed to  admit  a bounded
$H^{\infty}$-calculus on $L_p$.  Quite recently, non-linear SPDE of type \eqref{eqn h} with non-linear term $A(U)$ in place of $AU$ was studied in  \cite{roeckner} in the  Gelfand triple setting with the restriction $\alpha<1$ and $\beta<(\alpha+1/2)\vee 1$.

With regard to  equation \eqref{eqn 07.16.2},
an $L_2$-theory  was introduced  in \cite{chen2015fractional} for the equation driven only by Wiener processes, and the result of \cite{chen2015fractional}  was extended in \cite{kim16timefractionalspde} for $p\geq 2$.   The zero initial condition is assumed in both  \cite{chen2015fractional} and \cite{kim16timefractionalspde}. 

Actually, equation   \eqref{eqn 07.16.2} can be written in the integral form like \eqref{eqn h}, and  it is much more general than \eqref{eqn h} in the sense that it involves multiplicative noises and random operators depending also on $(t,x)$ together with  non-zero initial data.  We do not impose unnecessary algebraic conditions on $\alpha,\beta_1,\beta_2$, and most importantly  our equation is driven by more general processes, that is L\'evy processes. However, our results do not cover those in \cite{desch2013maximal, desch2008stochastic, desch2011p, roeckner} because the operator $A$ can belong to quite large  class of  operators.

For the deterministic counterpart of our result we refer e.g to \cite{dong2019lp,kim17timefractionalpde,zacher2006quasilinear}. We also refer to \cite{chen2014lp,kim2014sobolev,kry99analytic} for the classical case $\alpha=\beta_1=\beta_2=1$. 

This article is organized as follows. In Section 2 we introduce stochastic calculus related to L\'evy processes, preliminary results on the fractional calculus, and some properties of the solution space,  and  we 
 present our main result, Theorem \ref{thm 10.08.10:57}. 
 In Section 3 we use Littlewood-Paley theory to obtain key estimates for solutions.  In Section 4 we prove our main result.   In Section 5 we give an application to SPDEs driven by L\'evy space-time white noise.

Finally we introduce   notation used in this article. We use ``:='' to denote
a definition. As usual, $\mathbb{R}^{d}$ stands for the $d$-dimensional
Euclidean space of points $x=(x^{1},\ldots,x^{d})$.  $\bN$  denotes the set of natural numbers and $\bN_+=\{0\} \cup \bN$.  For $i=1,2,\cdots,d$ and  multi-index $\mathfrak{a}=(\mathfrak{a}_{1},\cdots,\mathfrak{a}_{d})$, where $\mathfrak{a}_{i} \in \bN_+$, we set
\begin{equation*}
D_{i}u=u_{x^{i}}=\frac{\partial}{\partial x^{i}}u,\quad D^{\mathfrak{a}}u=D^{\mathfrak{a}_{1}}_{1}\dots D^{\mathfrak{a}_{d}}_{d}u,\quad 
|\mathfrak{a}|=\mathfrak{a}_{1}+\dots+\mathfrak{a}_{d}.
\end{equation*} 
We also use  $D^{m}_{x}$ or $D^m$ to denote  arbitarry $m$-th order partial derivative with respect to $x$.  For $a,b\in \bR$,  $a\vee b:=\max(a,b)$ and $a^+:=a \vee 0$. By $\cF(f)$ or $\hat{f}$ we  denote  the Fourier transform of $f$.  $\Ccinf(\bR^d)$  denotes the set of infinitely differentiable functions with compact support in $\bR^d$,  $\cS(\R^{d})$ is the class of Schwartz functions on $\bR^d$, and $\cD=\cD(\bR^d)$ is the class of tempered distributions.  For $p\in[1,\infty]$, a measure space $(X,\cA,\mu)$, a normed vector space $B$ with norm $\|\cdot\|_{B}$,  $L_{p}(X,\cA;B)$ is the set of $B$-valued $\bar{\cA}$-measurable functions $f$ satisfying
\begin{equation*}
\|f\|_{L_{p}(X,\cA;B)}=\left(\int_{X}\|f(x)\|^{p}_{B}d\mu\right)^{1/p},
\end{equation*}
where $\bar{\cA}$ is the completion of $\cA$ with respect to $\mu$. 
We say $U$ is a version of $V$   if $\|U-V\|_B=0$ $\mu$-almost everywhere.  
If we write $C=C(a,b,\cdots)$, this means
that the constant $C$ depends only on $a,b,\cdots$.  Throughout
the article, for functions depending on $(\omega,t,x)$, the argument
$\omega \in \Omega$ will be usually omitted.

\mysection{main results}
First we introduce some definitions and facts related to the fractional calculus. For more detail, see e.g. \cite{baleanu12fractional,podlubny98fractional,richard14fractional,samko93fractional}.
For $\alpha>0$ and $\varphi\in L_1((0,T))$, the Riemann-Liouville fractional integral  of order $\alpha$ is defined by
$$ I_t^\alpha\varphi(t) := \frac{1}{\Gamma(\alpha)}\int_0^t(t-s)^{\alpha-1}\varphi(s)ds,\quad  t\leq T.
$$
For any $p\in[1,\infty]$,  we easily have
\begin{equation}\label{Jensen_fractional_integral}
\|I^\alpha \varphi \|_{L_p((0,T))}\leq C(\alpha,p,T)\| \varphi \|_{L_p((0,T))}.
\end{equation}
It is also easy to check that  if $\varphi \in L_p((0,T); B)$ and $\alpha>1/p$ then   $I^{\alpha}_t \varphi(t)$ is a continuous function satisfying $I^{\alpha}_t\varphi(0)=0$.   In particular if $\varphi$ is bounded, then $I^{\alpha}_t \varphi(t)$ is  continuous  for any $\alpha>0$.  The similar  statements hold if $\varphi(t)$ is an $L_p(\bR^d)$-valued (or Banach space-valued) function.

Let $n$ be the integer such that $n-1\leq \alpha<n$. If $\varphi$ is $(n-1)$-times differentiable, and $(\frac{d}{dt})^{n-1}I_{t}^{n-\alpha}\varphi$ is absolutely continuous on $[0,T]$, then the Riemann-Liouville fractional derivative $D_t^\alpha \varphi$ and  the Caputo fractional derivative $\partial_t^\alpha\varphi$ are defined by
\begin{equation*}
D_{t}^{\alpha}\varphi:=\left(\frac{d}{dt}\right)^{n}\left(I_{t}^{n-\alpha}\varphi\right),
\end{equation*}
\begin{equation}\label{Caputo_Riemann}
\partial^{\alpha}_{t}\varphi :=D^{\alpha}_{t}\left(\varphi(t)-\sum_{k=0}^{n-1}\frac{t^{k}}{k !}\varphi^{(k)}(0)\right).
\end{equation}
One can easily check that  for any $\alpha,\beta \geq0$, 
\begin{equation}
   \label{eqn 05.15.1}
I^{\alpha+\beta}\varphi(t)=I^{\alpha}I^{\beta}\varphi(t), \quad
D^\alpha D^\beta\varphi = D^{\alpha+\beta}\varphi,
\end{equation}
and
\begin{equation}\label{eqn 07.12.2}
D^\alpha I^\beta \varphi = \begin{cases} D^{\alpha-\beta}\varphi &\mbox{ if }\alpha>\beta \\ I^{\beta-\alpha}\varphi &\mbox{ if }\alpha\leq\beta\end{cases}.
\end{equation}

For $p>1$ and $\gamma \in \bR$, let
$H_{p}^{\gamma}=H_{p}^{\gamma}(\mathbb{R}^{d})$ denote  the class of all
tempered distributions $u$  on $\mathbb{R}^{d}$ such that
\begin{equation}
        \label{eqn norm}
\| u\| _{H_{p}^{\gamma}}:=\|(1-\Delta)^{\gamma/2}u\|_{L_{p}}<\infty,
\end{equation}
where
$$
(1-\Delta)^{\gamma/2} u = \cF^{-1} \left((1+|\xi|^2)^{\gamma/2}\cF (u) \right).
$$
The action of $u$ on $\phi\in\cS(\R^{d})$, which is denoted by $(u,\phi)$, is defined by
\begin{equation}
   \label{eqn 5.12.2}
(u,\phi):=((1-\Delta)^{\gamma/2}u, (1-\Delta)^{-\gamma/2}\phi).
\end{equation}
The number $\gamma$ is related to the regularity  of elements in above spaces. For instance,  if $\gamma=0, 1,2,\cdots$, then we have
$$
H^{\gamma}_p=\{u: D^{\mathfrak{a}} u\in L_p(\bR^d), \, \,\, |\mathfrak{a}|\leq \gamma\}, \quad \quad H^{-\gamma}_p=(H^{\gamma}_{p/{(p-1)}})^*.
$$
Let  $l_2$ denote the set of all sequences $a=(a^1,a^2,\cdots)$ such that
$$|a|_{l_{2}}:=\left(\sum_{k=1}^{\infty}|a^{k}|^{2}\right)^{1/2}<\infty.
$$
By $H_{p}^{\gamma}(l_{2})=H_{p}^{\gamma}(\bR^d,l_2)$  we denote the class of all $l_2$-valued
tempered distributions $v=(v^1,v^2,\cdots)$ on $\mathbb{R}^{d}$ such that
$$
\|v\|_{H_{p}^{\gamma}(l_{2})}:=\||(1-\Delta)^{\gamma/2}v|_{l_2}\|_{L_{p}}<\infty.
$$
Also we write  $h=(h^{1},\dots,h^{d_{1}})\in H^{\gamma}_{p}(l_{2},d_{1})$ if $h^1,h^2, \cdots, h^{d_1}$ are $l_2$-valued functions satisfying
$$
\|h\|_{H^{\gamma}_{p}(l_{2},d_{1})}:=\sum_{r=1}^{d_{1}}\|h^{r}\|_{H^{\gamma}_{p}(l_{2})}<\infty.
$$
Let $(\Omega,\mathscr{F},P)$ be a complete probability
space and $\{\mathscr{F}_{t},t\geq0\}$ be an increasing filtration of
$\sigma$-fields $\mathscr{F}_{t}\subset\mathscr{F}$, each of which
contains all $(\mathscr{F},P)$-null sets. We assume that
 independent families of one-dimensional
Wiener processes $\{W_{t}^{k}\}_{k\in\mathbb{N}}$ and $d_{1}$-dimensional L\'evy processes $\{Z_{t}^{k}\}_{k\in\mathbb{N}}$  relative to the
filtration $\{\mathscr{F}_{t},t\geq0\}$ are  given on $\Omega$.
By $\cP$ we denote the predictable $\sigma$-field generated by $\mathscr{F}_{t}$, i.e.,
$\cP$ is the smallest $\sigma$-field containing sets of the type $A \times (s,t]$, where $s<t$ and $A \in \mathscr{F}_s$.

For $p>1$ and $\gamma \in \bR$ denote
\begin{equation*}
\begin{gathered}
\mathbb{H}_{p}^{\gamma}(T):=L_{p}\left(\Omega\times
(0,T),\mathcal{P};H_{p}^{\gamma}\right),\quad\mathbb{L}_{p}(T)=\mathbb{H}_{p}^{0}(T),
\\
\mathbb{H}_{p}^{\gamma}(T,l_{2}):=L_{p}\left(\Omega\times(0,T),\mathcal{P};H_{p}^{\gamma}(l_{2})\right),\quad\mathbb{L}_{p}(T,l_{2})=\mathbb{H}_{p}^{0}(T,l_{2}),
\\
\mathbb{H}^{\gamma}_{p}(T,l_{2},d_{1}):=L_{p}\left(\Omega\times(0,T),\mathcal{P};H_{p}^{\gamma}(l_{2},d_{1})\right),\quad\mathbb{L}_{p}(T,l_{2},d_1)=\mathbb{H}_{p}^{0}(T,l_{2},d_{1}).
\end{gathered}
\end{equation*}

\begin{remark}
The spaces $\mathbb{H}_{p}^{\gamma}(T)$ are used  for (real-valued) solutions. The spaces $\mathbb{H}_{p}^{\gamma}(T,l_{2})$ and   $\mathbb{H}_{p}^{\gamma}(T,l_{2},d_1)$ are used for the integrands in the stochastic integrals against one-dimensional Wiener processes and $d_1$-dimensional L\'evy  processes, respectively. See \eqref{eqn 01.30.15:52}.
\end{remark}

 For
 $t\geq0$ and $A\in\cB(\R^{d_{1}}\setminus\{0\})$,  denote
\begin{equation*}
\begin{aligned}
&N^{k}(t,A):=\# \{0\leq s\leq t:\Delta Z^{k}_{s}:=Z^{k}_{s}-Z^{k}_{s-}\in A\}
\\
&\tilde{N}^{k}(t,A):=N^{k}(t,A)-t\nu^{k}(A),
\end{aligned}
\end{equation*}
where $\nu^{k}(A):=\bE N^{k}(1,A)$ is the L\'evy measure of $Z^{k}_{t}$. 
Set
\begin{equation*}
m_{p}(k):=\left(\int_{\R^{d_{1}}}|z|^{p}\nu^{k}(dz)\right)^{1/p}.
\end{equation*}
If $m_2(k)<\infty$, then by the L\'evy-It\^o decomposition, there exist a vector $a^{k}=(a^{1k},\dots,a^{d_{1}k})$, a non-negative definite $d_{1}\times d_{1}$ matrix $b^{k}$, and $d_{1}$-dimensional Wiener process $\tilde{W}^{k}_{t}$ such that
\begin{equation*}
Z^{k}_{t}=a^{k}t+b^{k}\tilde{W}^{k}_{t}+\int_{\bR^{d_1}}z\tilde{N}^{k}(t,dz)=:a^{k}t+b^{k}\tilde{W}^{k}_{t}+\tilde{Z}^k_t
 \end{equation*}
(i.e. $Z^{rk}_{t}=a^{rk}t+\sum_{l=1}^{d_{1}}b^{rlk}\tilde{W}^{l}_{t}
+\int_{\bR^{d_1}}z^{r}\tilde{N}^{k}(t,dz)$).

In this article, we assume the following.

\begin{assumption}\label{asm 07.10.1}
\begin{enumerate}[(i)]

\item $p\in[2,\infty)$ and 
\begin{equation*}
m_{p}:=\sup_{k}\big(m_{2}(k)\vee m_{p}(k)\big)<\infty.
\end{equation*}

\item  For each $k$, $Z^{k}_{t}=(Z^{1k},\dots,Z^{d_{1}k})$ is a pure jump process  with no diffusion and drift  parts (i.e. $a^k=0$ and $b^{k}=0$).

\end{enumerate}
\end{assumption}

\begin{remark}
We pose Assumption \ref{asm 07.10.1}   to have \eqref{levy}, for which the constant $C$ depends on $m_p$ (also see \eqref{eqn 5.19.5}).
Observe that for $p\geq 2$,
\begin{align*}
m^{p}_{p}(k)= \int_{\R^{d_1}} |z|^{p} \nu^{k}(dz).
\end{align*}
Hence, according to Assumption \ref{asm 07.10.1} (i), for larger $p$  we require more decay rate of the measure $\nu^{k}(dz)$ near infinity, that is as $|z|\to \infty$. Obviously Assumption \ref{asm 07.10.1} (i) holds if 
\begin{equation*}\label{eqn 01.05.20:14}
\sup_{t>0}\sup_{k}|\Delta Z^{k}_{t}|<\infty.
\end{equation*}
See \cite[Remark 2.2]{kim2014sobolev}. 
\end{remark}

\begin{remark}
\label{remark men}
One can also consider   equation \eqref{eqn 07.16.2} with general $Z^k_t$, without Assumption \ref{asm 07.10.1} (ii),  by rewriting  it into the form of the equation driven by a  set of Brownian motions $\{W^k_t\} \cup \{\tilde{W}^k_t\}$ and L\'evy processes  $\tilde{Z}^k_t$.  See (2.1) in \cite{kim2014sobolev} for detail. 
\end{remark}

\begin{remark}
If one only wants to prove the uniqueness and existence of $H^{\gamma+2}_{p}$-valued path-wise solution $u$, then Assumption \ref{asm 07.10.1} (i) can be replaced by the weaker condition that there is an integer $k_{0}\geq 1$ so that $\sup_{k\geq k_{0}}m_{p}(k)<\infty$. In particular, it can be completely dropped if only finitely many L\'evy processes appear in \eqref{eqn 07.16.2}. However, under this conditon we may have $\mathbb{E}\int_{0}^{T}\|u\|^{p}_{H^{\gamma+2}_{p}}dt=\infty$. See the proof of \cite[Theorem 4.9]{kim2012lp}.
\end{remark}

Due to the assumption $m_2(k)<\infty$,  $Z^{k}_{t}$ is a square integrable martingale, and  the stochastic integral against $Z^{rk}_t$ ($r=1,\cdots, d_1$)  can be easily understood as follows.  For functions $h$ of the type $h=\sum_{i=1}^m a_i1_{(\tau_i, \tau_{i+1}]}(t)$, 
where $\tau_i$ are bounded stopping times, $\tau_i\leq \tau_{i+1}$, and $a_i$ are bounded $\cF_{\tau_i}$-measurable random variables, we define
$$
(\Lambda h)_t:=\int^t_0 h \,dZ^{rk}_s:=\sum_{i=1}^m a_i (Z^{rk}_{\tau_{i+1} \wedge t}-Z^{rk}_{\tau_i \wedge t}).
$$
Then $\Lambda h$ becomes   a square integrable martingale with   c\`adl\`ag  sample paths, and one can easily check
$$
 \bE \sup_{t\leq T}|(\Lambda h)_t|^2\leq c_2(k) \|h\|^2_{L_2(\Omega \times [0,T])}.
$$
Therefore, the stochastic integral can be continuously extended to all $h\in L_2(\Omega\times [0,T], \cP; \bR)$, and $\int^t_0 h \,dZ^{rk}_t$ becomes a square integrable martingale with c\`adl\`ag sample paths. Furthermore, if $h_1=h_2$ in 
$L_2(\Omega\times [0,T], \cP; \bR)$, then 
$$\int^t_0 h_1 \,dZ^{rk}_t=\int^t_0 h_2 \,dZ^{rk}_t, \quad  \quad \forall\,\,  t\leq T \,\,(a.s.).
$$
This is  because  both are  c\`adl\`ag processes. For functions  $h=(h^{1},\dots,h^{d_{1}})\in L_{2}(\Omega\times[0,T],\cP;\R^{d_{1}})$ and $d_1$-dimentional L\'evy process $Z^k_t$ (here $k$ is fiexed), we define 
\begin{equation}\label{eqn 01.30.15:52}
\int^t_0 h \cdot dZ^k_t=\sum_{r=1}^{d_1}\int^t_0 h^r(s) dZ^{rk}_s, \quad t\leq T.
\end{equation}

\begin{remark}\label{rmk 07.10.2}
For any $h=(h^{1},\dots,h^{d_{1}})\in L_{2}(\Omega\times[0,T],\cP;\R^{d_{1}})$ with a predictable version $\bar{h}$,
$$
M^{k}_{t}=\int_{0}^{t}h \cdot dZ^{k}_{s}:=\sum_{r=1}^{d_1} \int^t_0 h^r dZ^{rk}_t=\sum_{r=1}^{d_1} \int^t_0 \bar{h}^r dZ^{rk}_t
$$
 is a square integrable martingale with the quadratic variation (see e.g. \cite{protter2005stochastic})
\begin{equation}\label{eqn 08.05.2}
\begin{aligned}
&[M^{k}]_{t}=\sum_{r,l=1}^{d_{1}}\int_{0}^{t}\int_{\R^{d_{1}}}z^{r}z^{l}\bar{h}^{r}_{s}\bar{h}^{l}_{s} N^{k}(ds,dz).
\end{aligned}
\end{equation}

\noindent
By \cite[Lemma 2.5]{chen2014lp} (or see \cite[Lemma 1] {mikulevicius2012l_p}) we have
\begin{equation}
\label{eqn 5.19.5}
\begin{aligned}
&\bE\left[\left(\sum_{k=1}^{\infty}\int_{0}^{T}\int_{\R^{d_{1}}}|z|^{2}|\bar{h}^{k}(s)|^{2}N^{k}(ds,dz)\right)^{p/2}\right]
\\
&\quad\leq  C(p,m_{p})\bE\left[\left(\int_{0}^{T}\sum_{k=1}^{\infty}|h^{k}(s)|^{2}ds\right)^{p/2}+\int_{0}^{T}\sum_{k=1}^{\infty}|h^{k}(s)|^{p}ds\right],
\end{aligned}
\end{equation}
where $|h^{k}(s)|^{2}=\sum_{r=1}^{d_{1}}|h^{rk}(s)|^{2}$. Since 
$$
\sum_{k=1}^{\infty}|a_{k}|^{p}\leq \left(\sum_{k=1}^{\infty}|a_{k}|^{2}\right)^{p/2},\quad \left(\int_{0}^{t}|h|ds\right)^{p/2}\leq t^{p/2-1}\int_{0}^{t}|h|^{p/2}ds,
$$ (recall that $p\geq2$), we have
\begin{equation}\label{eqn 07.16.1}
\bE\left[\left(\sum_{k=1}^{\infty}\int_{0}^{T}\int_{\R^{d_{1}}}|z|^{2}|\bar{h}^{k}(s)|^{2}N^{k}(ds,dz)\right)^{p/2}\right] \leq C \sum_{r=1}^{d_{1}}\bE \|h^{r}\|^{p}_{ L_{p}([0,T];l_{2})},
\end{equation}
where $C=C(p,m_p,d_{1},T)$. Therefore by the Burkholder-Davis-Gundy inequality, \eqref{eqn 08.05.2}, and \eqref{eqn 07.16.1},
\begin{equation}\label{eqn 08.05.3}
\begin{aligned}
\bE\left[\sup_{s\leq t}\left|\sum_{k=1}^{\infty}M^{k}_{s}\right|^{p}\right] \leq C \sum_{r=1}^{d_{1}} \bE\|h^{r}\|^{p}_{L_{p}([0,T];l_{2})}.
\end{aligned}
\end{equation}
\end{remark}

\begin{remark}\label{rmk 07.10.1}
(i) If $g\in\mathbb{H}^{\gamma}_{p}(T,l_{2})$, and $h\in\mathbb{H}^{\gamma}_{p}(T,l_{2},d_{1})$, then the series
\begin{equation*}
\begin{gathered}
\sum_{k=1}^{\infty}\int_{0}^{t}(g^{k}(s,\cdot),\phi)dW^{k}_{s},
\\
\sum_{k=1}^{\infty}\int_{0}^{t}(h^{k}(s,\cdot),\phi) \cdot dZ^{k}_{s}:=\sum_{k=1}^{\infty}\sum_{r=1}^{d_{1}}\int_{0}^{t}(h^{rk}(s,\cdot),\phi) dZ^{rk}_{s}
\end{gathered}
\end{equation*}
are well-defined due to Assumption \ref{asm 07.10.1} and Remark \ref{rmk 07.10.2}. Indeed,  using Remark \ref{rmk 07.10.2}  one can show (see \cite[Remark 3.2]{kry99analytic} for detail) 
\begin{equation*}
\sum_{r=1}^{d_{1}}\sum_{k=1}^{\infty}\int_{0}^{T}(h^{rk},\phi)^{2}ds\leq C(\phi,m_{p},d_{1},T)\|h\|^{p}_{\mathbb{H}^{\gamma}_{p}(T,l_{2},d_{1})}.
\end{equation*}
Therefore, the series 
$$
\sum_{k=1}^{\infty}\int_{0}^{t}(h^{k}(s,\cdot),\phi) \cdot dZ^{k}_{s}:=\sum_{k=1}^{\infty}\sum_{r=1}^{d_{1}}\int_{0}^{t}(h^{rk}(s,\cdot),\phi) dZ^{rk}_{s}
$$
converges in probability uniformly on $[0,T]$, and it is a square integrable martingale on $[0,T]$, which is c\`adl\`ag.  The same argument holds for $\sum_{k=1}^{\infty}\int_{0}^{t}(g^{k}(s,\cdot),\phi)dW^{k}_{s}$, which  is a continuous martingale on $[0,T]$.

(ii) The argument in (i) shows that if, for instance, $h_n \to h$ in $\mathbb{H}^{\gamma}_{p}(T,l_{2},d_{1})$, then as $n\to \infty$,
$$ \sum_{k=1}^{\infty}\int_{0}^{t}(h^{k}_n(s,\cdot),\phi) \cdot dZ^{k}_{s} \quad \to \quad \sum_{k=1}^{\infty}\int_{0}^{t}(h^{k}(s,\cdot),\phi) \cdot dZ^{k}_{s}
$$
in probability uniformly on $[0,T]$.
\end{remark}

We say that  $X_t=Y_t$ for almost all $t\leq T$ at once if 
$$
P\big(\{\omega: X_t(\omega)=Y_t(\omega), \,a.e. \,t\leq T\}\big)=1,
$$
and  $X_t=Y_t$ for all $t\leq T$ at once if 
$$
P\big(\{\omega: X_t(\omega)=Y_t(\omega), \forall \,t\leq T\}\big)=1.
$$

\begin{lemma}\label{lem 07.11.1}
Let $X^{k}_{t}=W^{k}_{t}$ or $X^{k}_{t}=Z^{rk}_{t}$, $r\in \{1,\cdots, d_1\}$,
and $h\in L_{2}(\Omega\times[0,T],\cP;l_{2})$.

(i)  Let $\alpha>0$ and  $h=(h^1,h^2,\cdots) \in \bL_2(T,l_2)$. Then 
 \begin{eqnarray*}
 I^{\alpha }\left(\sum_{k=1}^{\infty} \int^{\cdot}_0 h^k (s) dX^k_s \right)(t)
 &=&\sum_{k=1}^{\infty} I^{\alpha} \left( \int^{\cdot}_0 h^k (s) dX^k_s \right)(t)
 \end{eqnarray*}
for almost all $t\leq T$ at once. 

(ii) Under the assumptions in (i), 
$$
\sum_{k=1}^{\infty} I^{\alpha} \left( \int^{\cdot}_0 h^k (s) dX^k_s \right)(t)=\frac{1}{\alpha\Gamma(\alpha)}\sum_{k=1}^{\infty}\int^t_0(t-s)^{\alpha}h^k(s)dX^k_s
$$
 a.e. on  $\Omega\times [0,T]$.

(iii)  If $\alpha<1/2$, then 
\begin{eqnarray*}
D^{\alpha}_t \left(\sum_{k=1}^{\infty}\int_{0}^{\cdot}h^{k}(s)dX^{k}_{s}\right)(t)&=&\sum_{k=1}^{\infty} D^{\alpha}_t \left(\int_{0}^{\cdot}h^{k}(s)dX^{k}_{s}\right)(t)\\
&=&\frac{1}{\Gamma(1-\alpha)}\sum_{k=1}^{\infty}\int_{0}^{t}(t-s)^{-\alpha}h^{k}(s)dX^{k}_{s}
\end{eqnarray*}
 a.e. on  $\Omega\times [0,T]$.

\end{lemma}

\begin{proof}
 See Lemmas 3.1 and 3.3 in \cite{chen2015fractional} for (i) and (iii). Actually,   the case $X^{k}_{t}=W^{k}_{t}$ is proved in \cite{chen2015fractional}, and the same argument works for the general case  for $X^{k}_{t}=Z^{rk}_{t}$.

(ii) easily  follows from  the Stochastic Fubuni theorem (see \cite[Chapter IV, Theorem 65]{protter2005stochastic}).
\end{proof}

 Now we define spaces for initial data.  Set
$$
U^{\gamma+2}_{p}=L_{p}(\Omega,\mathscr{F}_{0};H^{\gamma+(2-2/\alpha p)^{+}}_{p}),
$$
and
\begin{equation}\label{eqn 08.09.4}
\begin{gathered}
V^{\gamma+2}_{p}= \begin{cases}L_{p}(\Omega,\mathscr{F}_{0};H^{\gamma+2-2/\alpha-2/\alpha p}_{p})&\quad \alpha>1+1/p
\\
L_{p}(\Omega,\mathscr{F}_{0};H^{\gamma+2-2/\alpha}_{p})&\quad 1<\alpha\leq 1+1/p.\end{cases}
\end{gathered}
\end{equation}
Note that if $\alpha>1+1/p$, then $2-2/\alpha-2/\alpha p>0$, and  $2-2/\alpha>0$ for any $\alpha>1$.

Fix a small   constant $\kappa>0$. Set
\begin{equation}\label{eqn 08.09.1}
\begin{gathered}
c_{0}:=1_{\beta_1>1/2} \frac{(2\beta_{1}-1)}{\alpha}+\kappa 1_{\beta_{1}=1/2},
\\
\bar{c}_{0}:=1_{\beta_2>1/p} \frac{(2\beta_{2}-2/p)}{\alpha}+\kappa 1_{\beta_{2}=1/p}.
\end{gathered}
\end{equation}
Note that $0\leq c_{0},\bar{c}_{0}< 2$,  $c_{0}=0$ if $\beta_1<1/2$,  and $\bar{c}_{0}=0$ if $\beta_{2}<1/p$. The constants $c_0$ and $\bar{c}_0$ are introduced to indicate the regularity (or differentiabliy) differences beween the solutions   free terms in stochastic parts of the equation, and the choice of these constants is optimal.  See Remark \ref{remark illu} for detail.

\begin{definition}\label{def 07.11.1} 
Let $p\geq 2$ and  $\gamma \in \bR$.  We write $u\in \cH^{\gamma+2}_p(T)$ if 
 $u\in \bH^{\gamma+2}_p(T)$ and there exist
$f\in\bH^{\gamma}_{p}(T),g\in \bH^{\gamma+c_0}_{p}(T,l_{2}),h\in \bH^{\gamma+\bar{c}_0}_{p}(T,l_{2},d_{1}),
 u_{0}\in U^{\gamma+2}_{p}$, and $v_{0}\in V^{\gamma+2}_p$ such that $u$ satisfies
\begin{equation}\label{eqn 07.10.2}
\begin{aligned}
&\partial^{\alpha}_{t}u(t,x)=f(t,x)+\partial^{\beta_{1}}_{t}\sum_{k=1}^{\infty}\int_{0}^{t}g^{k}(s,x)dW^{k}_{s}+\partial^{\beta_{2}}_{t}\sum_{k=1}^{\infty}\int_{0}^{t}h^{k} (s,x)   \cdot dZ^{k}_{s},\quad t\in(0,T]
\\
& u(0,\cdot)=u_{0},\quad  1_{\alpha>1}\partial_tu(0,\cdot)=1_{\alpha>1}v_{0}
\end{aligned}
\end{equation}
 in the sense of distributions. In other words, 
 for any $\phi\in\cS(\R^{d})$,  the equality
 \begin{equation}\label{eqn 07.12.2}
\begin{aligned}
(u(t)-u_{0}-tv_{0}1_{\alpha>1},\phi)&=I^{\alpha}_{t}(f,\phi)+\sum_{k=1}^{\infty}I^{\alpha-\beta_{1}}_{t}\int_{0}^{t}(g^{k}(s),\phi)dW^{k}_{s}
\\
&\quad\quad+\sum_{k=1}^{\infty}I^{\alpha-\beta_{2}}_{t}\int_{0}^{t}(h^{k}(s),\phi) \cdot dZ^{k}_{s}
\end{aligned}
\end{equation}
holds  a.e.  on  $\Omega\times [0,T]$,  (here $I^{\alpha-\beta_{i}}_{t}:=D^{\beta_{i}-\alpha}_{t}$ if $\beta_{i}>\alpha$). 
  \end{definition}
 
 \begin{remark} 
Note that, since $\beta_1<\alpha+1/2$ and $\beta_2<\alpha+1/p$, 
the right hand side of \eqref{eqn 07.12.2} makes sense due to  Lemma \ref{lem 07.11.1}.   \end{remark}

 \begin{remark} \label{remark con}
 If $\beta_2 > \alpha+1/p$ then  \eqref{eqn 07.12.2} does not make sense. For simplicity, let $u_0=v_0=0, f=0$ and $g=0$. Then taking $I^{\beta_2-\alpha}_t$ to  \eqref{eqn 07.12.2} we get
 $$
 I^{\beta_2-\alpha}_t(u(t),\phi)=\sum_{k=1}^{\infty}\int^t_0 (h^k(s),\phi)  \cdot dZ^k_s.
 $$
Since $(u(t),\phi)\in L_p([0,T])$ (a.s.) and $\beta_2-\alpha>1/p$, the left hand side above is a continuous process. However, the right hand side is only c\`adl\`ag process. The necessity of condition $\beta_1<\alpha+1/2$ can be derived similarly, and is explained in detail in \cite{chen2015fractional}.
 \end{remark}

Due to Lemma \ref{lem 07.11.1} (iii),  if $\beta_1<1/2$ or $\beta_2<1/2$, then the the expression in \eqref{eqn 07.10.2} is not unique, that is, there can be other triple $(f,g,h)$ such that  \eqref{eqn 07.10.2} holds in the sense of distributions. 

To define a norm in $\mathcal{H}^{\gamma+2}_{p}(T)$,  we introduce the space 
 $$\mathbb{F}^{\gamma}_{p}(T):=\mathbb{H}^{\gamma}_{p}(T)\times\mathbb{H}^{\gamma+c_{0}}_{p}(T,l_{2})
 \times\mathbb{H}^{\gamma+\bar{c}_{0}}_{p}(T,l_{2},d_{1}),$$ and for a triple $(f,g,h)\in\mathbb{F}^{\gamma}_{p}(T)$,
we define 
$$
\|(f,g,h)\|_{\mathbb{F}^{\gamma}_{p}(T)}=\|f\|_{\mathbb{H}^{\gamma}_{p}(T)}+\|g\|_{\mathbb{H}^{\gamma+c_{0}}_{p}(T,l_{2})}+\|h\|_{\mathbb{H}^{\gamma+\bar{c}_{0}}_{p}(T,l_{2},d_{1})}.
$$
\begin{definition}\label{def 11.02.15:06}
For $u\in \cH^{\gamma+2}_p(T)$, we define
\begin{align*}
\|u\|_{\cH^{\gamma+2}_p(T)}&= \|u\|_{\mathbb{H}^{\gamma+2}_{p}(T)}  + \|u(0)\|_{U^{\gamma+2}_{p}}+1_{\alpha>1}\|\partial_tu(0)\|_{V^{\gamma+2}_{p}}
+ \inf \left\{ \|(f,g,h)\|_{\mathbb{F}^{\gamma}_{p}(T)} \right\},
\end{align*}
where the infimum is taken for all triples $(f,g,h)\in\mathbb{F}^{\gamma}_{p}(T)$ such that \eqref{eqn 07.10.2} holds in the sense of distributions.
\end{definition}

In the following proposition, we address that our definition for \eqref{eqn 07.10.2} is equivalent to that of \cite[Definition 2.5]{kim16timefractionalspde}.

\begin{prop}
\label{prop}
Let $u\in \bH^{\gamma+2}_p(T)$,  $u_{0}\in U^{\gamma+2}_{p}$, $v_{0}\in V^{\gamma+2}_p$, and
$(f,g,h)\in\mathbb{F}^{\gamma}_{p}(T)$.  
Then the following  are equivalent;

(i) $u\in \cH^{\gamma+2}_p(T)$ and \eqref{eqn 07.10.2} holds  with $u_{0},v_{0}$, and  triple $(f,g,h)$ in the sense of Definition \ref{def 07.11.1}.

(ii) For any constant  $\Lambda$ such that 
 $$
 \Lambda \geq \max(\alpha,\beta_1,\beta_2) \quad \text{and}\quad \Lambda>\frac{1}{p},
 $$
 $I^{\Lambda-\alpha}_tu$  has an $H^{\gamma}_{p}$-valued c\`adl\`ag  version in  $\bH^{\gamma}_p(T)$, still denoted by  $I^{\Lambda-\alpha}_{t}u$, such that   for any $\phi\in \cS(\bR^d)$, 
the equality  
\begin{equation}\label{eqn 07.10.3}
\begin{aligned}
&(I^{\Lambda-\alpha}_{t}u-I^{\Lambda-\alpha}_{t}(u_{0}+tv_{0}1_{\alpha>1}),\phi)
\\
&\quad=I^{\Lambda}_{t}(f,\phi)+\sum_{k=1}^{\infty}I^{\Lambda-\beta_{1}}_{t}\int_{0}^{t}(g^{k}(s,\cdot),\phi)dW^{k}_{s}+\sum_{k=1}^{\infty}I^{\Lambda-\beta_{2}}_{t}\int_{0}^{t}(h^{k}(s,\cdot),\phi) \cdot dZ^{k}_{s}
\end{aligned}
\end{equation}
holds for all $t\in[0,T]$ at once. Moreover, in this case it holds that
\begin{equation}\label{eqn 10.05.14:19}
\begin{aligned}
\mathbb{E}\sup_{t\leq T}\|I^{\Lambda-\alpha}_{t}u\|^{p}_{H^{\gamma}_{p}} \leq &C \Big( \mathbb{E}\|u_{0}\|^{p}_{H^{\gamma}_{p}}  + 1_{\alpha>1}\mathbb{E}\|v_{0}\|^{p}_{H^{\gamma}_{p}}
\\
&\,\quad +\|f\|^{p}_{\mathbb{H}^{\gamma}_{p}(T)}+\|g\|^{p}_{\mathbb{H}^{\gamma}_{p}(T,l_{2})}+\|h\|^{p}_{\mathbb{H}^{\gamma}_{p}(T,l_{2},d_{1})} \Big),
\end{aligned}
\end{equation}
where the constant $C$ depends only on $\alpha,\beta_{1},\beta_{2},d,d_{1},p,\gamma, \Lambda$ and $T$.

 \end{prop}

\begin{proof}
Considering $(1-\Delta)^{\gamma/2}u$ in place of $u$, we may assume $\gamma=0$.

(i) Suppose  \eqref{eqn 07.10.3} holds for all $t$ at once.  Then by applying $D^{\Lambda-\alpha}_{t}$ to \eqref{eqn 07.10.3}
 and using  \eqref{eqn 05.15.1}, \eqref{eqn 07.12.2}, and Lemma \ref{lem 07.11.1}, we find that  \eqref{eqn 07.10.2}  holds for a.e. on $\Omega\times [0,T]$.
\vspace{1mm}

(ii) Suppose \eqref{eqn 07.10.2} holds a.e. on $\Omega\times [0,T]$.   Note that  $I^{\Lambda-\alpha}_t(u_0+1_{\alpha>1}t v_0)$ is a continuous $L_p$-valued process, and it satisfies
$$
\mathbb{E}\sup_{t\leq T} \|I^{\Lambda-\alpha}_t(u_0+1_{\alpha>1}t v_0)\|^{p}_{L_{p}} \leq C(T) (\|u_{0}\|^{p}_{L_{p}}+1_{\alpha>1}\|v_{0}\|^{p}_{L_{p}}).
$$
Hence we may assume $u_0=v_0=0$. 

Take a nonnegative funtion $\zeta\in\Ccinf(\R^{d})$ with unit integral. For each $n>0$, define $\zeta_{n}(x)=n^{-d}\zeta(nx)$.  For any tempered distribution $v$, define $v^{(n)}(x):=v\ast\zeta_{n}(x)$. Then $v^{(n)}$ is infinitely differentiable function with respect to $x$. 
Plugging  $\phi=\zeta_{n}(\cdot-x)$ in \eqref{eqn 07.10.2} and applying $I^{\Lambda-\alpha}_t$ to both sides of \eqref{eqn 07.10.2}, for each $x$ we get 
\begin{eqnarray}
\nonumber 
(I^{\Lambda-\alpha}_{t}(u)^{(n)}) (t,x)&=&(I^{\Lambda}_t f^{(n)})(t,x)+\sum_{k=1}^{\infty}I^{\Lambda-\beta_{1}}_{t}\int_{0}^{t}(g^{k})^{(n)}(s,x) dW^{k}_{s}\\
&&+\sum_{k=1}^{\infty}I^{\Lambda-\beta_{2}}_{t}\int_{0}^{t}(h^{k})^{(n)}(s,x) \cdot dZ^{k}_{s} \label{eqn 5.13.1}
\end{eqnarray}
a.e. on $\Omega\times [0,T]$. Note that since $\Lambda>1/p$, $I^{\Lambda}_t f^{(n)}$ is a continuous $L_p$-valued process. Also,  the stochastic integrals
$$
\sum_{k=1}^{\infty}\int_{0}^{t}(g^{k})^{(n)}(s,x) dW^{k}_{s}, \quad
\sum_{k=1}^{\infty}\int_{0}^{t}(h^{k})^{(n)}(s,x) \cdot dZ^{k}_{s}
$$
are $L_{p}$-valued c\`adl\`ag  processes, and in particular  they are bounded on $[0,T]$ (a.s.).  Therefore, the right hand side
 of \eqref{eqn 5.13.1} is an $L_{p}$-valued c\`adl\`ag  process, and  consequently the left hand side has an $L_{p}$-valued c\`adl\`ag version,  still denoted by  
 $I^{\Lambda-\alpha}_{t}(u)^{(n)}$.
 
 By  \eqref{Jensen_fractional_integral} with $p=\infty$ and  \eqref{eqn 08.05.3},
\begin{equation*}
\begin{aligned}
&\bE\sup_{t\leq T}\left\|I^{\Lambda-\beta_{2}}_{t}\sum_{k=1}^{\infty}\int_{0}^{t}(h^{k})^{(n)}(s,\cdot) \cdot dZ^{k}_{s}\right\|^{p}_{L_{p}}\\
&\leq C\int_{\R^{d}}\bE\sup_{t\leq T}\left|I^{\Lambda-\beta_{2}}_{t}\sum_{k=1}^{\infty}\int_{0}^{t}(h^{k})^{(n)} (s,x) \cdot dZ^{k}_{s}\right|^{p}dx
\\
&\leq C\int_{\R^{d}}\bE\sup_{t\leq T}\left|\sum_{k=1}^{\infty}\int_{0}^{t}(h^{k})^{(n)} (s,x) \cdot dZ^{k}_{s}\right|^{p}dx
 \leq C\bE \int_{0}^{T}\|h^{(n)}(s,\cdot)\|^{p}_{L_{p}(l_{2},d_{1})}ds.
\end{aligned}
\end{equation*}
We  handle two other terms on the right hand side of \eqref{eqn 5.13.1}  similarly,  and get
\begin{eqnarray}
\nonumber
&& \bE\sup_{t\leq T}\left\|I^{\Lambda-\alpha}(u)^{(n)}(t,\cdot)\right\|^{p}_{L_{p}}\\
&&\leq C \big(\|f^{(n)}\|^p_{\bL_p(T)}+  \|g^{(n)}\|^{p}_{\bL_{p}(T,l_{2})}+\|h^{(n)}\|^{p}_{\bL_{p}(T,l_{2},d_{1})}\big). \label{eqn 5.13.5}
\end{eqnarray}
Considering \eqref{eqn 5.13.5} corresponding to  $I^{\Lambda-\alpha}_{t}(u)^{(n)}- I^{\Lambda-\alpha}_{t}(u)^{(m)}$,  we find that  $I^{\Lambda-\alpha}_{t}(u)^{(n)}$  is a Cauchy sequence in $L_{p}(\Omega;D([0,T];L_{p}))$, where $D([0,T];L_{p})$ is a space of $L_{p}$-valued c\`adl\`ag  functions. Let $w$ denote the limit  in this space.  Then
since $I^{\Lambda-\alpha}_{t}(u)^{(n)} \to I^{\Lambda-\alpha}u$ in $\bL_p(T)$,  we conclude $w=I^{\Lambda-\alpha}_tu$ a.e on  $\Omega\times [0,T]$, and $w$ is 
an $L_{p}$-valued c\`adl\`ag  version of $I^{\Lambda-\alpha}_tu$. This proves that   \eqref{eqn 07.10.3} holds for all $t$ at once because both sides are read-valued c\`adl\`ag  processes. Also we easily obtain  \eqref{eqn 10.05.14:19} from \eqref{eqn 5.13.5} and the lemma is proved.
\end{proof}

\begin{theorem}\label{thm 07.22.1}

Let $p\geq 2, \gamma \in \bR$ and $T \in (0,\infty)$.

(i) For any $\nu\in\R$, the map $(1-\Delta)^{\nu/2}:\mathcal{H}^{\gamma+2}_{p}(T)\to \mathcal{H}^{\gamma-\nu+2}_{p}(T)$ is an isometry.

(ii)  $\mathcal{H}^{\gamma+2}_{p}(T)$ is a Banach space  with the norm in Definition \ref{def 11.02.15:06}.  

(iii) Suppose that $u\in \mathcal{H}^{\gamma+2}_{p}(T)$ satisfies \eqref{eqn 07.10.2}  with a triple $(f,g,h)\in\mathbb{F}^{\gamma}_{p}(T)$.    Then for any $t\leq T$, 
\begin{eqnarray} \nonumber
\|u\|^{p}_{\bH^{\gamma}_{p}(t)}&\leq& C \int_{0}^{t}(t-s)^{\theta-1}\Big(\|f\|^{p}_{\bH^{\gamma}_{p}(s)}+\|g\|^{p}_{\bH^{\gamma}_{p}(s,l_{2})}+\|h\|^{p}_{\bH^{\gamma}_{p}(s,l_{2},d_{1})} \Big)ds
\\
&&+C(\mathbb{E}\|u_{0}\|^{p}_{H^{\gamma}_{p}}+1_{\alpha>1}\mathbb{E}\|v_{0}\|^{p}_{H^{\gamma}_{p}}),
 \label{eqn 07.26.3}
\end{eqnarray} 
where $\theta:=\min\{\alpha,2(\alpha-\beta_{1})+1, p(\alpha-\beta_{2})+2\}$, and the constant $C$ depends only on $\alpha,\beta_{1},\beta_{2},d,d_{1},p$ and $T$.
\end{theorem}

\begin{proof}
(i) This easily follows from the fact that 
$(1-\Delta)^{\nu/2}: H^{\gamma+2}_p \to H^{\gamma-\nu+2}_p$
is an isometry. 

(ii)  We only prove the completeness. Suppose that $u_{n}$ is a Cauchy sequence in $\mathcal{H}^{\gamma+2}_{p}(T)$ with  $u_{n}(0)=u^{n}_{0}$, and $1_{\alpha>1}\partial_tu_{n}(0)=1_{\alpha>1}v^{n}_{0}$. Since it is enough to prove there exists a convergent subsequence, by taking suitable subsequence we may assume that $\|u_{n+1}-u_{n}\|_{\mathcal{H}^{\gamma+2}_{p}(T)}\leq 2^{-n}$ for each $n\in\bN$. By the definition, for each $n\in\bN$, there exists $(\tilde{f}^{n+1},\tilde{g}^{n+1},\tilde{h}^{n+1})\in\mathbb{F}^{\gamma}_{p}(T)$  with which $u_{n+1}-u_{n}$ (in place of $u$) satisfies \eqref{eqn 07.10.2}, and
\begin{eqnarray} \nonumber
&&\|u_{n+1}-u_{n}\|_{\mathbb{H}^{\gamma+2}_{p}(T)}+\|u^{n+1}_{0}-u^{n}_{0}\|_{U^{\gamma+2}_{p}}
+1_{\alpha>1}\|v^{n+1}_{0}-v^{n}_{0}\|_{V^{\gamma+2}_{p}}\\
&&+\|(\tilde{f}^{n},\tilde{g}^{n},\tilde{h}^{n})\|_{\mathbb{F}^{\gamma}_{p}(T)}
  \leq \|u_{n+1}-u_{n}\|_{\mathcal{H}^{\gamma+2}_{p}(T)}+2^{-n} \leq 2^{-n+1}. \label{eqn 11.08.15:36}
\end{eqnarray}
We take a triple $(\tilde{f}^{1},\tilde{g}^{1},\tilde{h}^{1})\in\mathbb{F}^{\gamma}_{p}(T)$ such that $u_1$  satisfies \eqref{eqn 07.10.2} with this triple, and define
\begin{equation*}
\begin{gathered}
(f_{n},g_{n},h_{n})=\sum_{k=1}^{n}(\tilde{f}^{k},\tilde{g}^{k},\tilde{h}^{k}),   \qquad   (f,g,h)=\sum_{k=1}^{\infty}(\tilde{f}^{k},\tilde{g}^{k},\tilde{h}^{k}),
\\
u=\sum_{k=1}^{\infty}\left(u_{j+1}-u_{j}\right) + u_{1}.
\end{gathered}
\end{equation*}
Then, it is obvious that $u_{n}$  satisfies \eqref{eqn 07.10.2} with the triple $(f_{n},g_{n},h_{n})$, and 
\begin{eqnarray*}
&&\|u-u_{n}\|_{\mathbb{H}^{\gamma+2}_{p}(T)} +\|u_{0}-u^{n}_{0}\|_{U^{\gamma+2}_{p}}+1_{\alpha>1}\|v_{0}-v^{n}_{0}\|_{V^{\gamma+2}_{p}}
\\
&&\qquad +\|(f-f_{n},g-g_{n},h-h_{n})\|_{\mathbb{F}^{\gamma}_{p}(T)} 
\\
&& \leq\sum_{k=n+1}^{\infty} \Big( \|u_{k}-u_{k-1}\|_{\mathbb{H}^{\gamma+2}_{p}(T)}+\|u^{k}_{0}-u^{k-1}_{0}\|_{U^{\gamma+2}_{p}}
\\
&&\quad\quad\quad\quad\quad\quad+1_{\alpha>1}\|v^{k}_{0}-v^{k-1}_{0}\|_{V^{\gamma+2}_{p}}+\|(\tilde{f}^{k},\tilde{g}^{k},\tilde{h}^{k})\|_{\mathbb{F}^{\gamma}_{p}(T)}  \Big)
\\
&& \leq \sum_{k=n+1}^{\infty} 2^{-k+1}.
\end{eqnarray*}
Hence, to prove $u_{n}$ converges to $u$ in $\mathcal{H}^{\gamma+2}_{p}(T)$,  it is enough to show  $u$ satisfies  \eqref{eqn 07.10.2} with the triple $(f,g,h)$. This can be easily proved using Proposition \ref{prop} and Remark \ref{rmk 07.10.1} (ii).

(iii) We repeat the proof of  \cite[Theorem 2.1]{kim16timefractionalspde} which treats the case $h=0$.  Note first that by  the result of (i) we may assume $\gamma=0$.  

We take notation from the proof of Proposition \ref{prop}.   Then, from \eqref{eqn 07.12.2}, for each $x\in \bR^{d}$ we get
\begin{eqnarray}
\nonumber
u^{(n)}(t,x)&=&(u_{0})^{(n)}(x)+1_{\alpha>1}t(v_{0})^{(n)}(x) \\
&&+I^{\alpha}_{t}f^{(n)}(t,x)+\sum_{k=1}^{\infty}I^{\alpha-\beta_{1}}_{t}\int_{0}^{t}(g^{k})^{(n)}(s,x)\,dW^{k}_{s}\\
\nonumber
&&
+\sum_{k=1}^{\infty}I^{\alpha-\beta_{2}}_{t}\int_{0}^{t}(h^{k})^{(n)}(s,x) \cdot dZ^{k}_{s}
\label{eqn 05.21.5}
\end{eqnarray}
a.e. on $\Omega \times [0,T]$. Since $u^{(n)} \to u$ in $\bL_p(T)$, to prove \eqref{eqn 07.26.3}, it is enough to estimate $\|u^{(n)}\|_{\bL_p(t)}$. For this, we only estimate the last  term in \eqref{eqn 05.21.5} because  other terms are estimated in the proof of \cite[Theorem 2.1]{kim16timefractionalspde}.
   By Lemma \ref{lem 07.11.1}, for each $x\in\bR^{d}$ we have
$$
\left(\sum_{k=1}^{\infty}I^{\alpha-\beta_{2}}_{t}\int_{0}^{\cdot}(h^{k})^{(n)}(r,x)  \cdot dZ^{k}_{r} \right)(s)=c(\alpha,\beta_{2})\sum_{k=1}^{\infty}\int_{0}^{s}(s-r)^{\alpha-\beta_{2}}(h^{k})^{(n)}(r,x) \cdot dZ^{k}_{r}
$$
a.e. on $\Omega\times[0,t]$. Let $\bar{h}$ be a predictable version of $h$, then by the Burkerholder-Davis-Gundy inequality, \eqref{eqn 07.16.1} and Fubini's theorem, we have
\begin{equation*}
\begin{aligned}
&\left\|\sum_{k=1}^{\infty}I^{\alpha-\beta_{2}}_{t}\int_{0}^{\cdot}(h^{k})^{(n)}(s,x) \cdot dZ^{k}_{s}\right\|^{p}_{\bL_{p}(t)} 
\\
&\quad\leq C\int_{\R^{d}}\int_{0}^{t} \bE\left[\left(\sum_{k=1}^{\infty}\int_{0}^{s}\int_{\R}|z|^{2}|(s-r)^{\alpha-\beta_{2}}(\bar{h}^{k})^{(n)}(r,x)|^{2}N^{k}(dr,dz)\right)^{p/2}\right]dsdx
\\
&\quad\leq C \int_{0}^{t}\int^s_0 (s-r)^{p(\alpha-\beta_2)} \|h^{(n)}(r)\|^p_{\bL_p(l_2,d_1)} drds
\\
&\quad \leq C \int_{0}^{t}(t-s)^{p(\alpha-\beta_{2})+1}\|h(s)\|^{p}_{\bL_{p}(s,l_{2},d_{1})}ds \leq C \int_{0}^{t}(t-s)^{\theta-1}\|h(s)\|^{p}_{\bL_{p}(s,l_{2},d_{1})}ds.
\end{aligned}
\end{equation*}
Other terms in the right hand side of \eqref{eqn 05.21.5} can be handled similarly, and these yield inequality  \eqref{eqn 07.26.3} with $u^{(n)}$.  This is  enough because $u^{(n)}\to u$ in $\bL_p(t)$. 
\end{proof}

Take $\kappa'\in(0,1)$, and for $r\geq0$, set
\begin{equation*}
B^{r}=\begin{cases} L_{\infty}(\R^{d}) &\mbox{if}\quad r=0
\\
C^{r-1,1}(\R^{d}) & \mbox{if}\quad r=1,2,3,\dots
\\
C^{r+\kappa'}(\R^{d})&\mbox{otherwise},
\end{cases}
\end{equation*}
where $C^{r+\kappa'}(\R^{d})$ and $C^{r-1,1}(\R^{d})$ are H\"older space and Zygmund space respectively. We use $B^{r}(l_{2})$ for $l_{2}$-valued analogue. It is known (see e.g. \cite[Lemma 5.2]{kry99analytic}) that  for $u\in H^{\gamma}_{p}$
\begin{equation}\label{eqn 08.06.1}
\begin{gathered}
\|au\|_{H^{\gamma}_{p}}\leq C(d,p,\kappa',\gamma)|a|_{B^{|\gamma|}}\|u\|_{H^{\gamma}_{p}},
\\
\|bu\|_{H^{\gamma}_{p}(l_{2})}\leq C(d,p,\kappa',\gamma)|b|_{B^{|\gamma|}(l_{2})}\|u\|_{H^{\gamma}_{p}}.
\end{gathered}
\end{equation}

\begin{assumption}\label{asm 07.16.1}

(i) All the coefficients are  $\cP\otimes\cB(\R^{d})$-measurable functions.

(ii) The coefficients $\mu^i, \nu, \bar{\mu}^{ir}, \bar{\nu}^r$ are $l_2$-valued functions, where $i=1,2,\cdots,d$ and $r=1,2,\cdots, d_1$. 

(iii)
There exists a constant $0<\delta<1$ so that for any $(\omega,t,x)$
\begin{equation}\label{eqn 07.22.2}
\begin{gathered}
\delta|\xi|^{2}\leq a^{ij}(t,x)\xi^{i}\xi^{j}\leq \delta^{-1}|\xi|^{2},\quad \forall\xi\in\R^{d}.
\end{gathered}
\end{equation}

(iv)
The coefficients $a^{ij}(\omega,t,x)$ is uniformly continuous in $(t,x)$, uniformly on $\Omega$.  

(v)
For each $\omega,t, i,j,r$,
\begin{eqnarray*}
&& |a^{ij}(t,\cdot)|_{B^{|\gamma|}} + |b^{i}(t,\cdot)|_{B^{|\gamma|}}+ |c(t,\cdot)|_{B^{|\gamma|}} +
|\mu^{i}(t,\cdot)|_{B^{|\gamma+c_{0}|}(l_{2})}\\
&&+ |\nu(t,\cdot)|_{B^{|\gamma+c_{0}|}(l_{2})}+
|\bar{\mu}^{ir}(t,\cdot)|_{B^{|\gamma+\bar{c}_{0}|}(l_{2})} + |\bar{\nu}^{r}(t,\cdot)|_{B^{|\gamma+\bar{c}_{0}|}(l_{2})} \leq \delta^{-1}.
\end{eqnarray*}

(vi)
 $\mu^{i}=0$ if $\beta_1\geq 1/2+\alpha/2$, and  $\bar{\mu}^{ir}=0$ if $\beta_{2}\geq 1/p+\alpha/2$.

\end{assumption}

\begin{remark}
(i) Assumption \ref{asm 07.16.1} (iv) is posed due to \eqref{eqn 08.06.1}. That is, this assumption is needed  for the coefficients  to become point-wise multipliers in appropriate Banach spaces.

(ii)  Assumption \ref{asm 07.16.1} (vi) is to avoid having too high order derivatives of solution in the stochastic parts of the equation. Even if $\alpha=\beta_1=\beta_2=1$, one cannot have derivatives of order greater than $1$ in the stochastic parts. See \cite{kry99analytic}. 
\end{remark}

Below we use  notation $f(u),g(u)$, and $h(u)$ to denote $f(\omega, t,x,u),g(\omega, t,x,u)$, and $h(\omega, t,x,u)$ respectively.

\begin{assumption}\label{asm 10.08.10:29}
(i)$f,g$ and $h$ are $\cP\times\cB(\bR^{d+1})$ measurable, and for any $u\in\mathbb{H}^{\gamma+2}_{p}(T)$,
$$
f(u)\in\mathbb{H}^{\gamma}_{p}(T), \quad g(u)\in\mathbb{H}^{\gamma+c_{0}}_{p}(T,l_{2}), \quad h(u)\in\mathbb{H}^{\gamma+\bar{c}_{0}}_{p}(T,l_{2},d_{1}).
$$

(ii) For any $\varepsilon>0$, there exists a constant $K=K(\varepsilon)>0$ so that
\begin{equation*}
\begin{aligned}
&\|f(t,u)-f(t,v)\|_{H^{\gamma}_{p}}+\|g(t,u)-g(t,v)\|_{H^{\gamma+c_{0}}_{p}(l_{2})}
+\|h(t,u)-h(t,v)\|_{H^{\gamma+\bar{c}_{0}}_{p(l_{2},d_{1})}}
\\
&\quad \leq \varepsilon \|u-v\|_{H^{\gamma+2}_{p}}+K \|u-v\|_{H^{\gamma}_{p}}
\end{aligned}
\end{equation*}
 for any $\omega,t$, and $u,v\in H^{\gamma+2}_{p}$ .
\end{assumption}

Here is the main result of this article.

\begin{theorem}\label{thm 10.08.10:57}
Let $\gamma\in\R$, $p\geq2$, and $T<\infty$. Suppose  Assumption \ref{asm 07.16.1} and Assumption \ref{asm 10.08.10:29} hold and
$$
 \alpha\in(0,2), \quad \beta_{1}<\alpha+1/2, \quad \beta_{2}<\alpha+1/p.
 $$
Then for any $u_{0}\in U^{\gamma+2}_{p}$, $v_{0}\in V^{\gamma+2}_{p}$,  equation \eqref{eqn 07.16.2}
has a unique solution $u$ in the class $\mathcal{H}^{\gamma+2}_{p}(T)$ in the sense of Definition \ref{def 07.11.1}. Moreover, 
\begin{equation}\label{eqn 10.22.14:25}
\begin{aligned}
  \|u\|_{\mathcal{H}^{\gamma+2}_{p}(T)}  &\leq C \Big(\|u_{0}\|_{U^{\gamma+2}_{p}}+1_{\alpha>1}\|v_{0}\|_{V^{\gamma+2}_{p}}+\|f(0)\|_{\bH^{\gamma}_{p}(T)}
\\
&\quad \quad \quad +\|g(0)\|_{\bH^{\gamma+c_{0}}_{p}(T,l_{2})}+\|h(0)\|_{\bH^{\gamma+\bar{c}_{0}}_{p}(T,l_{2},d_{1})}\Big),
\end{aligned}
\end{equation}
where the constant $C$ depends only on $\alpha,\beta_{1},\beta_{2},d,d_{1},p,\delta,\gamma,\kappa$, and $T$.
\end{theorem}

\begin{remark}
 If $\alpha=\beta_1=1$ and $h(0)=0$ then Theorem \ref{thm 10.08.10:57} follows from \cite[Theorem 5.2]{kry99analytic}.

\end{remark}

\begin{remark}
\label{remark illu}
  To  explain main steps of the proof and give an  easy exposition of our result,  we consider the model equation
\begin{equation}\label{eqn 01.10.14:35}
\partial^{\alpha}_{t} u = \Delta u + f + \partial^{\beta_{1}}_{t}\int_{0}^{t} g dW_{s} + \partial^{\beta_{2}}_{t}\int_{0}^{t} h dZ_{s}, \quad t>0
\end{equation}
given with zero initial condition, where $W_{t}$ and $Z_{t}$ are one-dimensional independent Wiener process and L\'evy process.  A main step for Theorem \ref{thm 10.08.10:57} (for $\gamma=0$) is  to show
\begin{equation}\label{eqn 01.10.15:10}
\|\Delta u \|_{\mathbb{L}_{p}(T)} \leq C\left ( \|f\|_{\mathbb{L}_{p}(T)} + \|(-\Delta)^{c_{0}/2}g\|_{\mathbb{L}_{p}(T)} + \|(-\Delta)^{\bar{c}_{0}/2}h\|_{\mathbb{L}_{p}(T)} \right).
\end{equation}
Here $\bL_p(T)=L_p(\Omega\times (0,T); L_p(\bR^d))$ and $C$ depends only on $p,m_p,\alpha,\beta_1,\beta_2$.  To be more specific, we can  control the norm $\|u\|_{\bL_p(T)}$ by the right-hand side of \eqref{eqn 10.22.14:25} using the integration representaion of solution, and estimate of $\|u\|_{\bL_p(T)}$ and \eqref{eqn 01.10.15:10} imply \eqref{eqn 10.22.14:25}
for the model equation if $\gamma=0$. For general $\gamma\in \bR$, we consider $v:=(1-\Delta)^{\gamma/2}u$ and its corresponding SPDE (i.e. \eqref{eqn 01.10.14:35} with $v,(1-\Delta)^{\gamma/2}f,(1-\Delta)^{\gamma/2}g,(1-\Delta)^{\gamma/2}h$ in place of $u,f,g,h$ respectively). The estimate for $v$ when $\gamma=0$ gives \eqref{eqn 10.22.14:25} for  arbitrary $\gamma \in \bR$.  The generalization to the non-linear equation with measurable coefficients is done based on certain perturbation arguments.

Speaking of  \eqref{eqn 01.10.15:10}, we prove this  based mainly on Littlewood-Paley theory in harmonic analysis. Recall that if $\beta_1>1/2$ and $\beta_2>1/p$, then $c_0:=(2\beta_1-1)\alpha^{-1}$ and $\bar{c}_0:=(2\beta_2-2/p)\alpha^{-1}$.  Actually these  are only possible constants for which \eqref{eqn 01.10.15:10} holds for any $T>0$. This can be proved using a scaling argument.  For simplicity let $f=h=0$, Then using the relation
$\partial^{\alpha}_{t}(u(c\cdot))(t) = c^{\alpha}\partial^{\alpha}_{t}u(ct)$,
we see that $u_{c}(t,x) = u(c^{2/\alpha}t,cx)$ satisfies
$$
\partial^{\alpha}_{t}u_{c} = \Delta u_{c} + \partial^{\beta_{1}}_{t}\int_{0}^{t} g_{c}(s,x) dW^{c}_{s}, \quad t<c^{-2/\alpha}T,
$$
where 
$$
W^{c}_{t} = c^{-1/\alpha}W_{c^{2/\alpha}t}, \quad  \quad g_{c}(t,x) = c^{2-\frac{2\beta_{1}-1}{\alpha}}g(c^{2/\alpha}t,cx).
$$
Applying  \eqref{eqn 01.10.15:10} for $u_c$ and $c^{-2/{\alpha}}T$, and  using the change of variables $c^{2/{\alpha}}t \to t$,   we can easily  get 
\begin{align*}
\|\Delta u\|_{\mathbb{L}_{p}(T)} 
\leq C 
  c^{c_{0}-\frac{2\beta_{1}-1}{\alpha}} \|(-\Delta)^{c_{0}/2}g\|_{\mathbb{L}_{p}(T)}, \quad \forall c>0.
\end{align*}
This makes sense only if $c_0=(2\beta_1-1)/{\alpha}$.  The scaling argumet for the  case $h\neq 0$ can be similarly done. For  detail, see the proof of Corollary 2.21 in \cite{ kim2014sobolev}.

\end{remark}

\begin{remark} If $\alpha\in (0,1]$ then Assumption \ref{asm 07.16.1} (iv) can be relaxed and replaced by the uniformly continuity in $x$, uniformly on $\Omega\times [0,T]$.
Assumption \ref{asm 07.16.1} (iv) is inherited from a result on the deterministic equation, \cite[Theorem 2.10]{kim17timefractionalpde}. However, if $\alpha\in (0,1]$ then the continuity in $t$ can be completely dropped for the deterministic equation (see e.g. \cite{dong2019lp}). 
\end{remark}

\mysection{Key estimates}

In this section we study the convolution operators of the type 
$$ ((-\Delta)^{a}D^b_t \,p)* f,  
$$
 where $a,b\in \bR$, $p(t,x)$ is the fundamental solution to the time fractional heat equation $\partial^{\alpha}_tu=\Delta u$, and $(-\Delta)^{a}$ is the fractional Laplacian of order $a$ defined by
 $$
 (-\Delta)^{a}f(x)=\cF^{-1}\{|\cdot|^{2a}\cF(f)(\cdot)\}(x).
 $$
 To explain the necessity of such study,  let us consider 
 $$
 \partial^{\alpha}_t u=\Delta u+\partial^{\beta}_t\int^t_0 h(s) \, dZ_t, \,\,\, t>0 \quad ; \quad u(0)=1_{\alpha>1}u_t(0)=0,
 $$
where $Z_t$ is a L\'evy process. It turns out that for the solution $u$ and $c\geq 0$  we have 
$$
\|(-\Delta)^{c/2} u\|^p_{\bL_p(T)}\leq C  \Big\|\int^t_0 \left|\left((-\Delta)^{c/2}D^{\beta-\alpha}_t p\right) \ast h(s) \right|^pds \Big\|_{L_1(\Omega\times [0,T]; L_1(\bR^d))}.
$$
Thus, for the estimations of solutions,  we need to handle the right hand side of the above inequality.
If  non-zero initial condition is given, this also leads to the similar situation.

Below, to state our main theorems of this section, we introduce the Besov space. 
We fix   $\Psi\in \cS(\bR^d)$ such that its Fourier transform $\hat{\Psi}(\xi)$ has support in a strip $\{\xi\in\R^{d}|\frac{1}{2}\leq |\xi|\leq 2\}$, $\hat{\Psi}(\xi)>0$ for $\frac{1}{2}<|\xi|<2$, and 
\begin{equation*}
\sum_{j\in\bZ}\hat{\Psi}(2^{-j}\xi)=1 \quad \textrm{for} \quad \xi\neq0.
\end{equation*}
Define 
\begin{equation*}
\begin{gathered}
\hat{\Psi}_{j}(\xi)=\hat{\Psi}(2^{-j}\xi), \quad j=\pm1,\pm2,\dots ,
\\
\hat{\Psi}_{0}(\xi)=1-\sum_{j=1}^{\infty}\hat{\Psi}_{j}(\xi).
\end{gathered}
\end{equation*}
For distributions (or functions) $f$,  we denote $f_{j}:=\Psi_{j}\ast f$.

It is known that if $u\in H^{\gamma}_{p}$, then
\begin{equation}
                          \label{rmk 08.27.1}
                          \|u\|_{H^{\gamma}_{p}}\sim \left(\|u_{0}\|_{L_{p}}+\|(\sum_{j=1}^{\infty}2^{\gamma j}|u_{j}|^{2})^{1/2}\|_{L_{p}}\right).
\end{equation}
For $1<p<\infty$, $s\in\R$, we define Besov space $B_{p}^{s}=B_{p}^{s}(\R^{d})$ as the collection of all tempered distributions $u$ such that 
\begin{equation*}
 \|u\|_{B_{p}^{s}}:=\|u_{0}\|_{L_{p}}+\left[\sum_{j=1}^{\infty}2^{spj}\|u_{j}\|^{p}_{L_{p}}\right]^{1/p}<\infty.
\end{equation*}

\begin{remark}\label{fractional_laplacian_besov}
It is well known (e.g. \cite{bergh2012interpolation, triebel1995interpolation}) that $\Ccinf(\R^{d})$ is dense in $B^{s}_{p}$,  the inclusion
 $B^{s_{2}}_{p}\subset B^{s_{1}}_{p}$ holds  for  $s_{1}\leq s_{2}$, and 
\begin{equation*}
\begin{gathered}
H^{s}_{p} \subset B^{s}_{p},  \quad 2\leq p<\infty.
\end{gathered}
\end{equation*}
Furthermore,  $(-\Delta)^{\frac{\gamma}{2}}$ is a bounded operator from $B^{s+\gamma}_{p}$ to $B^{s}_{p}$, and $(1-\Delta)^{\gamma/2}$ is an isometry from $B^{s+\gamma}_{p}$ to $B^{s}_{p}$ and from $H^{s+\gamma}_{p}$ to $H^{s}_{p}$.
\end{remark}

\vspace{2mm}

Now, let  $0<\alpha<2$ and  $p(t,x)$ be the fundamental solution to the equation
\begin{equation}\label{equation_with_nonzero_initial_data}
\partial_{t}^{\alpha}u=\Delta u ,\quad u(0,x)=u_{0}(x), \quad 1_{\alpha>1}\partial_{t}u(0,x)=0.
\end{equation}
That is, $p(t,x)$ is the function so that, under appropriate smoothness assumption on $u_0$,  $u=p(t,\cdot)*u_0$ is the solution to \eqref{equation_with_nonzero_initial_data}.
For $\beta<\alpha+\frac{1}{2}$, we define

\begin{equation*}
q_{\alpha,\beta}(t,x) = \begin{cases}  I_t^{\alpha-\beta}p(t,x)~&\alpha\geq\beta, \\ D_t^{\beta-\alpha}p(t,x)~&\alpha<\beta. \end{cases}
\end{equation*}

Below we list some properties of $p$ and $q_{\alpha,\beta}$.

\begin{lemma}\label{properties_of_p_q}
Let  $0<\alpha<2$, $\beta<\alpha+\frac{1}{2}$, and $\gamma\in[0,2)$.

(i) For any $t>0$, and $x\neq 0$,
\begin{equation*}\label{eqn 07.18.2}
\partial^{\alpha}_{t}p(t,x)=\Delta p(t,x),\quad \frac{\partial}{\partial t}p(t,x)=\Delta q_{\alpha,1}(t,x),
\end{equation*}
and $\frac{\partial}{\partial t}p(t,x) \to 0$ as $t\downarrow 0$. Moreover, $\frac{\partial}{\partial t}p(t,\cdot)$ is integrable in $\R^{d}$ uniformly on $t\in[\varepsilon,T]$ for any $\varepsilon>0$.

(ii) For $f\in C^{\infty}_{c}(\R^{d})$, the convolution
\begin{equation*}
\int_{\R^{d}}p(t,x-y)f(y)dy
\end{equation*}
converges to $f(x)$ uniformly as $t\downarrow 0$.

(iii) For any  $m\in \bN_+$, there exist constants $c=c(\alpha,d,m)$ and $C=C(\alpha,d,m)$ such that if  $R:=|x|^{2}t^{-\alpha} \geq 1$, then
\begin{equation}\label{bound_of_p_1}
|D^{m}_{x}p(t,x)|\leq C t^{-\frac{\alpha (d+m)}{2}}\exp{\{-ct^{-\frac{\alpha}{2-\alpha}}|x|^{\frac{2}{2-\alpha}}\}},
\end{equation}
and if  $R<1$, then
\begin{equation}\label{bound_of_p_2}
\begin{aligned}
|D^{m}_{x}p(t,x)|\leq C  |x|^{-d-m}(R+R|\log{R}|1_{d=2,m=0}+R^{1/2}1_{d=1,m=0}).
\end{aligned}
\end{equation}

(iv) It holds that 
\begin{equation}\label{fourier_fo_q}
\cF\{D^{\sigma}_{t}q_{\alpha,\beta}(t,\cdot)\}=t^{\alpha-\beta-\sigma}E_{\alpha,1+\alpha-\beta-\sigma}(-t^{\alpha}|\xi|^{2}),
\end{equation}
 where $E_{a,b}$, $a>0$ is the Mittag-Leffler function defined as
\begin{equation*}
E_{a,b}(z)=\sum_{k=0}^{\infty}\frac{z^{k}}{\Gamma (ak+b)}, \quad z\in\bC.
\end{equation*}

(v) For any $\sigma\geq 0$, there exists a constant $C=C(\alpha,\beta,\sigma,\gamma,d)$ such that
\begin{equation}\label{bounds_of_q_1}
|D^{\sigma}_{t}(-\Delta)^{\gamma/2}q_{\alpha,\beta}(1,x)|+|D^{\sigma}_{t}(-\Delta)^{\gamma/2}\partial_{t}q_{\alpha,\beta}(1,x)|\leq C (|x|^{-d+2-\gamma}\wedge |x|^{-d-\gamma})
\end{equation}
if $d\geq2$, and
\begin{equation}\label{bounds_of_q_2}
\begin{aligned}
&|D^{\sigma}_{t}(-\Delta)^{\gamma/2}q_{\alpha,\beta}(1,x)|+|D^{\sigma}_{t}(-\Delta)^{\gamma/2}\partial_{t}q_{\alpha,\beta}(1,x)| 
\\
&\quad\quad \leq C(|x|^{1-\gamma}(1+\log{|x|}1_{\gamma=1})\wedge|x|^{-1-\gamma})
\end{aligned}
\end{equation}
if $d=1$. 

(vi) For any $\sigma \geq 0$, the following scaling property holds:
\begin{equation}\label{scaling_of_q_2}
D^{\sigma}_{t}(-\Delta)^{\gamma/2}q_{\alpha,\beta}(t,x)=t^{-\sigma-\frac{\alpha (d+\gamma)}{2}+\alpha-\beta}(-\Delta)^{\gamma/2}q_{\alpha,\beta}(1,t^{-\frac{\alpha}{2}}x).
\end{equation}
\end{lemma}

\begin{proof}
For (i), (iv), (v), and (vi), see \cite[Lemma 3.1]{kim16timefractionalspde}. Also see \cite[Lemma 3.1]{kim17timefractionalpde} for (iii), and see \cite[Corollary 3.2]{kim16timefractionalspde} for (ii).
\end{proof}

\begin{lemma}\label{integral_rep_of_Mittag_Leffler}
Let $0<a<2$ and  $b<a+1$. Then there exist constants 
\begin{equation*}
\eta_{1}>0, \quad \eta_{2}\in \bR, \quad \eta_{3}\in (-1,1)
\end{equation*}
which depend only on $a$ such that for any $v>0$,
\begin{equation}\label{representation_of_E_a_b}
E_{a,b}(-v)=\frac{1}{\pi a}\int_{0}^{\infty}\frac{r^{\frac{1-b}{a}}\exp{(-r^{1/a}\eta_{1})}[r\sin{(\psi-\eta_{2}a)}+v\sin{(\psi)}]}{r^{2}+2rv\eta_{3}+v^{2}}dr,
\end{equation}
where $\psi=\psi(r)=r^{1/a}\sin{(\eta_{2})}+(\eta_{2}(a+1-b))$.
\end{lemma}
\begin{proof}
The proof is based on \cite[Chapter 4]{gorenflo2014mittag}. Since $0<a<2$,  we can choose a constant  $\eta$ satisfying $\frac{a}{2}\pi<\eta<(\pi\wedge a\pi)$. Then by using  formula (4.7.13) in \cite{gorenflo2014mittag}, for any $v>0$ and for any $0<\lambda<v$, we have
\begin{equation}\label{middle_in_representation_of_E_a_b}
\begin{aligned}
E_{a,b}(-v)=&\frac{1}{\pi a}\int_{\lambda}^{\infty}\frac{r^{\frac{1-b}{a}}\exp{(r^{1/a}\cos{(\eta/a)})}[r\sin{(\psi-\eta)}+v\sin{(\psi)}]}{r^{2}+2rv\cos{(\eta)}+v^{2}}dr
\\
&+\int_{-\eta}^{\eta}G(a,b,\lambda,\phi,v)d\phi,
\end{aligned}
\end{equation}
where 
$$
\psi=\psi(r)=r^{1/a}\sin{(\eta/a)}+(\eta(a+1-b)/a),
$$  
$$
G=\frac{\lambda^{1+(1-b)/a}}{2\pi a}\frac{\exp{(\lambda^{1/a}\cos{(\phi/a)})}e^{i\nu}}{\lambda e^{i\phi}+v},
$$
and $\nu=\lambda^{1/a}\sin{(\phi/a)}+\phi(1+(1-b)/a)$. Since $b-1<a$, by the dominated convergence theorem, if we let $\lambda\downarrow 0$, the second integral in \eqref{middle_in_representation_of_E_a_b} goes to zero. Also since $\frac{a}{2}\pi<\eta<(\pi\wedge a\pi)$, $\cos(\eta/a)$ has negative value, and $|\cos{(\eta)}|\neq1$. Therefore, as $\lambda$ goes to zero, the first integral in \eqref{middle_in_representation_of_E_a_b} converges to the integral over positive real line with the same integrand by the dominated convergence theorem.  Therefore, to finish the proof, it is enough to  take $\eta_{1}=-\cos{(\eta/a)}, \eta_{2}=\eta/a$ and $\eta_{3}=\cos{(\eta)}$.
\end{proof}

\begin{remark}
If $b=1$, then we have
\begin{equation}\label{representation_of_E_a_1}
E_{a,1}(-v)=\frac{\sin{a \pi}}{\pi}\int_{0}^{\infty}\frac{r^{a-1}}{r^{2a}+2r^{a} \cos{(a\pi)}+1}\exp{(-rv^{1/a})}rdr.
\end{equation}
(see e.g. \cite[Exercise 3.9.5]{gorenflo2014mittag}).
\end{remark}

\begin{lemma}\label{integral_rep_of_frac_q}
Let $\alpha\in (0,2)$ and $\beta<\alpha+1/2$. Then there exist constants $C$ and
\begin{equation*}
m_{1}>0, \quad m_{2}\in \bR, \quad m_{3}\in \bR,\quad m_{4}\in \bR,  \quad m_{5}\in (-1,1),
\end{equation*} 
depending only on $\alpha,\beta$, such that for any $\mu\in \bR$
\begin{equation}\label{representation_of_frac_Lap_of_q}
\begin{aligned}
&\cF\{(-\Delta)^{\mu/2}q_{\alpha,\beta}(t,\cdot)\}(\xi)
\\
&\quad=C|\xi|^{\mu+\frac{2\beta-2\alpha}{\alpha}}\int_{0}^{\infty}\frac{\exp{(-m_{1}t|\xi|^{\frac{2}{\alpha}}r )}[r^{\alpha}\sin{(\tilde\psi+m_{3})}+\sin{(\tilde\psi+m_{4})}]}{r^{2\alpha}-2r^{\alpha}m_{5}+1}r^{\beta-1}dr,
\end{aligned}
\end{equation}
where $\tilde\psi=\tilde\psi(r)=m_{2}t|\xi|^{\frac{2}{\alpha}}r$.
\end{lemma}
\begin{proof}
By the definition of fractional Laplacian and \eqref{fourier_fo_q}, we have
\begin{equation*}
\cF\{(-\Delta)^{\mu/2}q_{\alpha,\beta}(t,\cdot)\}(\xi)=t^{\alpha-\beta}|\xi|^{\mu}E_{\alpha,1+\alpha-\beta}(-t^{\alpha}|\xi|^{2}).
\end{equation*}
By \eqref{representation_of_E_a_b} with $a=\alpha,b=1+\alpha-\beta$, and the change of variables $r\to vr$,  for any  $v>0$ we have
\begin{equation*}
\begin{aligned}
E&_{\alpha,1+\alpha-\beta}(-v)
\\
&=\frac{1}{\pi \alpha}\int_{0}^{\infty}\frac{r^{\frac{\beta-\alpha}{\alpha}}\exp{(-r^{1/\alpha}\eta_{1})}[r\sin{(\psi-\eta_{2}\alpha)}+v\sin{(\psi)}]}{r^{2}+2rv\eta_{3}+v^{2}}dr
\\
&=\frac{1}{\pi \alpha}\int_{0}^{\infty}\frac{v^{\frac{\beta - \alpha}{\alpha}}r^{\frac{\beta - \alpha}{\alpha}}\exp{(-v^{1/\alpha}r^{1/\alpha}\eta_{1})}[r\sin{(\psi_{1}-\eta_{2}\alpha)}+\sin{(\psi_{1})}]}{r^{2}+2r\eta_{3}+1}dr,
\end{aligned}
\end{equation*}
where $\psi_{1}=\psi_{1}(r)=v^{1/\alpha}r^{1/a}\sin{(\eta_{2})}+\eta_{2}\beta$.
By the change of variables $r\to r^{\alpha}$, 
\begin{equation*}
\begin{aligned}
E&_{\alpha,1+\alpha-\beta}(-v)
\\
&=\frac{1}{\pi \alpha}\int_{0}^{\infty}\frac{v^{\frac{\beta - \alpha}{\alpha}}r^{\beta - \alpha}\exp{(-v^{1/\alpha}r\eta_{1})}[r^{\alpha}\sin{(\psi''-\eta_{2}\alpha)}+\sin{(\psi')}]}{r^{2\alpha}+2r^{\alpha}\eta_{3}+1}\alpha r^{\alpha-1}dr
\\
&=C\int_{0}^{\infty}v^{\frac{\beta - \alpha}{\alpha}}\frac{\exp{(-v^{1/\alpha}r\eta_{1})}[r^{\alpha}\sin{(\psi_{2}-\eta_{2}\alpha)}+\sin{(\psi_{2})}]}{r^{2\alpha}+2r^{\alpha}\eta_{3}+1} r^{\beta-1}dr,
\end{aligned}
\end{equation*}
where $\psi_{2}=\psi_{2}(r)=v^{1/\alpha}r\sin{(\eta_{2})}+\eta_{2}\beta$. Putting $v=t^{\alpha}|\xi|^{2}$, we have
\begin{equation*}
\begin{aligned}
\cF&\{(-\Delta)^{\mu/2}q_{\alpha,\beta}(t,\cdot)\}(\xi)
\\
&=C|\xi|^{\mu+\frac{2\beta-2\alpha}{\alpha}}\int_{0}^{\infty}\frac{\exp{(-\eta_{1}t|\xi|^{\frac{2}{\alpha}}r)}[r^{\alpha}\sin{(\psi_{3}-\eta_{2}\alpha)}+\sin{(\psi_{3})}]}{r^{2\alpha}+2r^{\alpha}\eta_{3}+1} r^{\beta-1}dr,
\end{aligned}
\end{equation*}
due to \eqref{fourier_fo_q} and \eqref{setting_of_a_b_c}, where $\psi_{3}=\psi_{3}(r)=\sin{(\eta_{2})}t|\xi|^{\frac{2}{\alpha}}r+\eta_{2}\beta$. 

Finally, for  \eqref{representation_of_frac_Lap_of_q} we take 
\begin{equation*}
m_{1}=\eta_{1},\quad m_{2}=\sin{(\eta_{2})},\quad m_{4}=\eta_{2}\beta,\quad m_{3}=m_{4}-\alpha\eta_{2}, \quad m_{5}=\eta_{3}.
\end{equation*}
The lemma is proved.
\end{proof}

For each $j=0,1,\dots$ and $c>0$, denote
\begin{equation}\label{def_of_q_j}
\begin{aligned}
q_{\alpha,\beta}^{c,j}(t,x)&=\Psi_{j}\ast (-\Delta)^{\frac{c}{2}}q_{\alpha,\beta}(t,x)
\\
&=\cF^{-1}\{\hat{\Psi}(2^{-j}\cdot)\cF\{(-\Delta)^{\frac{c}{2}}q_{\alpha,\beta}\}(t,\cdot)\}(x)
\\
&=2^{jd}\cF^{-1}\{\hat{\Psi}(\cdot)\cF\{(-\Delta)^{\frac{c}{2}}q_{\alpha,\beta}\}(t,2^{j}\cdot)\}(2^{j}x)
\\
&=: 2^{jd}\bar{q}^{c,j}_{\alpha,\beta}(t,2^{j}x).
\end{aligned}
\end{equation}

\begin{lemma}\label{bound_of_q_j}
Assume
\begin{equation}\label{setting_of_a_b_c}
p\geq 2,\quad 0< \alpha<2,  \quad \frac{1}{p} < \beta < \alpha+\frac{1}{p},
\end{equation}
denote
$c_1:=\frac{2(\alpha+1/p-\beta)}{\alpha} >0$. Then for any constants $\varepsilon,\delta$ satisfying
\begin{equation}\label{epsilon_delta}
\frac{1}{p}<\beta-\frac{\alpha}{2}\varepsilon, \quad  \beta-\alpha<\frac{1}{p}-\delta<\frac{1}{p},
\end{equation}
 we have
\begin{equation}\label{bound_of_q_j_3}
\|q_{\alpha,\beta}^{c_{1}+\varepsilon,j}(t,\cdot)\|_{L_{1}}\leq C (2^{\frac{2\delta}{\alpha}j+\varepsilon j}t^{-\frac{1}{p}+\delta} \wedge t^{-\frac{1}{p}-\frac{\alpha\varepsilon}{2}}),
\end{equation}
where  $C=C(\alpha,\beta,d,p,\varepsilon,\delta)$.
\end{lemma}
\begin{proof}
Put $c_2=c_1+\varepsilon$. Then by \eqref{epsilon_delta}, $0<c_2<2$.   Due to  \eqref{scaling_of_q_2}, we easily get
\begin{equation}
\label{eqn 5.16.6}
\|(-\Delta)^{\frac{c_2}{2}}q_{\alpha,\beta}(t,\cdot)\|_{L_{1}} \leq C t^{-\frac{\alpha c_2}{2}+\alpha-\beta}.
\end{equation}
Recall that the convolution operator is bounded in $L_p$ for any $p\geq 1$, that is $\|f \ast g\|_{L_p}\leq \|f\|_{L_1}\|g\|_{L_p}$. Thus, \eqref{eqn 5.16.6} together with the first equality in \eqref{def_of_q_j} yields
$$
\|q_{\alpha,\beta}^{c_{2},j}(t,\cdot)\|_{L_{1}}\leq C t^{-\frac{1}{p}-\frac{\alpha\varepsilon}{2}}.
$$  
This and the equality $\|q_{\alpha,\beta}^{c_{2},j}(t,\cdot)\|_{L_{1}}=\|\bar{q}_{\alpha,\beta}^{c_{2},j}(t,\cdot)\|_{L_{1}}$ show that it remains to prove 
$$
\|\bar{q}_{\alpha,\beta}^{c_{2},j}(t,\cdot)\|_{L_{1}}\leq C 2^{\frac{2\delta}{\alpha}j+\varepsilon j}t^{-\frac{1}{p}+\delta}.
$$
By definition  (see \eqref{def_of_q_j})
\begin{equation}
 \label{q}
\cF(\bar{q}_{\alpha,\beta}^{c_{2},j})(t,\xi)=\hat{\Psi}(\xi)\cF\{(-\Delta)^{\frac{c_2}{2}}q_{\alpha,\beta}\}(t,2^{j}\xi).
\end{equation}
  Thus
\begin{equation}\label{middle_of_bound_of_q_j}
\begin{aligned}
|\cF(\bar{q}_{\alpha,\beta}^{c_{2},j})(t,\xi)|&=|\hat{\Psi}(\xi)||\cF\{(-\Delta)^{\frac{c_2}{2}}q_{\alpha,\beta}(t,\cdot)\}(2^{j}\xi)|
\\
&\leq C 1_{\frac{1}{2}\leq |\xi|\leq 2}|\cF\{(-\Delta)^{\frac{c_2}{2}}q_{\alpha,\beta}(t,\cdot)\}(2^{j}\xi)|.
\end{aligned}
\end{equation}

\noindent
By \eqref{representation_of_frac_Lap_of_q} with $\mu=\frac{c_2}{2}$,
\begin{eqnarray}
                          \label{middle_of_bound_of_q_j_3}
&&|\cF\{(-\Delta)^{\frac{c_2}{2}}q_{\alpha,\beta}(t,\cdot)\}(2^j\xi)|\\
&&\leq C |2^j\xi|^{\frac{2}{\alpha p} +\varepsilon}\int_{0}^{\infty}\frac{\exp{(-m_{1}t|2^j\xi|^{\frac{2}{\alpha}}r )}(|r^{\alpha}\sin{(\psi+m_{3})}|+|\sin{(\psi+m_{4})}|)}{r^{2\alpha}-2r^{\alpha}m_{5}+1}r^{\beta-1}dr,    \nonumber
\end{eqnarray}
where $\psi =m_{2}t|2^j \xi|^{\frac{2}{\alpha}}r$. Note that for any polynomial $Q$ of degree $m$ and $c>0$, there exists a constant $C(c,m)$ such that 
\begin{equation}\label{middle_of_bound_of_q_j_4}
Q(r)\exp{(-cr)}\leq C r^{-\frac{1}{p}+\delta} \quad r>0.
\end{equation}
Applying this inequality with $Q(r)=1$ and $c=m_1$ to \eqref{middle_of_bound_of_q_j_3}, we have
\begin{eqnarray*}
|\cF(\bar{q}_{\alpha,\beta}^{c_{2},j})(t,\xi)| &\leq& C1_{\frac{1}{2}\leq |\xi|\leq 2} |2^{j}\xi|^{\frac{2}{\alpha p} +\varepsilon}
\\
&&\times \left( \int_{0}^{1}(t|2^{j}\xi|^{\frac{2}{\alpha}}r)^{-\frac{1}{p}+\delta}r^{\beta-1}dr+\int_{1}^{\infty}(t|2^{j}\xi|^{\frac{2}{\alpha}}r)^{-\frac{1}{p}+\delta}r^{\beta-1}r^{-2\alpha}dr \right)\\
&\leq& C 2^{\frac{2\delta}{\alpha}j+\varepsilon j}t^{-\frac{1}{p}+\delta}1_{\frac{1}{2}\leq |\xi|\leq 2}.
\end{eqnarray*}
For the second inequality above we used $\beta-\alpha<\frac{1}{p}-\delta<\beta$. 

Similarly, using \eqref{q}, \eqref{representation_of_frac_Lap_of_q} and following the above computations,   for any multi-index $\gamma$ we get
\begin{equation*}
|D_{\xi}^{\gamma}\cF(\bar{q}_{\alpha,\beta}^{c_{2},j})(t,\xi)| \leq C 2^{\frac{2\delta}{\alpha}j+\varepsilon j}t^{-\frac{1}{p}+\delta}1_{\frac{1}{2}\leq |\xi|\leq 2}.
\end{equation*}
 Hence, we have
\begin{equation*}
\begin{aligned}
\|\bar{q}_{\alpha,\beta}^{c_{2},j}(t,\cdot)\|_{L_{1}}&=\int_{\R^{d}}(1+|x|^{2d})^{-1}(1+|x|^{2d})|\bar{q}_{\alpha,\beta}^{c_{2},j}(t,x)|dx
\\
&\leq C \int_{\R^{d}}(1+|x|^{2d})^{-1}\sup_{\xi}|(1+\Delta_{\xi}^{d})\cF(\bar{q}_{\alpha,\beta}^{c_{2},j})(t,\xi))|dx
\\
&\leq C 2^{\frac{2\delta}{\alpha}j+\varepsilon j}t^{-\frac{1}{p}+\delta}.
\end{aligned}
\end{equation*}
For the first inequality above we used the fact that if $\cF(f)$ has  compact support, then 
$$|f(x)|=|\cF^{-1} (\cF(f))(x)|\leq \|\cF(f)\|_{L_1}\leq C \sup_{\xi} |\cF(f)|.
$$
The lemma is proved.
\end{proof}

The following result will be used later to study the regularity relation between the solutions and free terms in stochastic parts.

\begin{theorem}\label{Besov_bound}
Let  \eqref{setting_of_a_b_c} and \eqref{epsilon_delta} hold, and denote
$c_1:=\frac{2(\alpha+1/p-\beta)}{\alpha} $.
 Then there exists a constant $C$ depdending only on $\alpha,\beta,d,p, \varepsilon, \delta$, and $T$ such that for any $g\in\Ccinf((0,\infty)\times\R^{d})$\begin{equation}\label{middle_of_Besov_bound_5}
\begin{aligned}
\int_{0}^{T}\int_{0}^{t}\int_{\R^{d}} \left|(-\Delta)^{\frac{c_{1}+\varepsilon}{2}}q_{\alpha,\beta}(t-s,x)\ast g(s)(x) \right|^{p}dxdsdt
\leq C \int_{0}^{T}\|g(t,\cdot)\|_{B_{p}^{\varepsilon}}^{p}dt.
\end{aligned}
\end{equation}
\end{theorem}

\begin{proof}
Denote $c_2=c_1+\varepsilon$ and 
$Q(t,x):=(-\Delta)^{\frac{c_2}{2}}q_{\alpha,\beta}(t,x)$. 
  By  \eqref{rmk 08.27.1},
\begin{equation*}
\begin{aligned}
\int_{0}^{T}\int_{0}^{t}\int_{\R^{d}}|&Q(t-s) \ast g(s)(x)|^{p}dxdsdt
\\
&\leq C \int_{0}^{T}\int_{0}^{t}\int_{\R^{d}}|\Psi_{0}\ast (Q(t-s)\ast g(s))(x)|^{p}
\\
&\quad\quad\quad\quad\quad\quad\quad+\left(|\sum_{j=1}^{\infty}|\Psi_{j}\ast (Q(t-s)\ast g(s))(x)|^{2}\right)^{p/2}dxdsdt.
\end{aligned}
\end{equation*}
Note that 
\begin{equation}\label{middle_of_Besov_bound_6}
\begin{aligned}
\hat{\Psi}_{j}&=\hat{\Psi}_{j}(\hat{\Psi}_{j-1}+\hat{\Psi}_{j}+\hat{\Psi}_{j+1}), \quad j=1,2,\dots,
\\
\hat{\Psi}_{0}&=\hat{\Psi}_{0}(\hat{\Psi}_{0}+\hat{\Psi}_{1}).
\end{aligned}
\end{equation} 
Using this and the relation $\cF(f_{1}\ast f_{2})=\cF(f_{1})\cF(f_{2})$, we get
\begin{equation*}
\begin{gathered}
\sum_{j=1}^{\infty}|\Psi_{j}\ast (Q(t-s)\ast g(s))(x)|^{2}
=\sum_{j=1}^{\infty}|\sum_{i=j-1}^{j+1}Q_{i}(t-s)\ast g_{j}(s)(x)|^{2},
\\
|\Psi_{0}\ast (Q(t-s)\ast g(s))(x)|=|q^{c_{2},0}_{\alpha,\beta}(t-s)\ast g_{0}(s)(x)+q^{c_{2},1}_{\alpha,\beta}(t-s,\cdot)\ast g_{0}(s)(x)|.
\end{gathered}
\end{equation*}

Therefore, 
\begin{eqnarray}
&&\int_{0}^{T}\int_{0}^{t}\int_{\R^{d}}|Q(t-s) \ast g(s)(x)|^{p}dxdsdt   \nonumber \\
&\leq &C \int_{0}^{T}\int_{0}^{t}\int_{\R^{d}}|q^{c_{2},0}_{\alpha,\beta}(t-s)\ast g_{0}(s)(x)|^{p}dxdsdt \nonumber
\\
&&+C\int_{0}^{T}\int_{0}^{t}\int_{\R^{d}}|q^{c_2,1}_{\alpha,\beta}(t-s)\ast g_{0}(s)(x)|^{p}dxdsdt \nonumber
\\
&&+C\int_{0}^{T}\int_{0}^{t}\int_{\R^{d}}\big(\sum_{j=1}^{\infty}|\sum_{i=j-1}^{j+1}q^{c_2,i}_{\alpha,\beta}(t-s)\ast g_{j}(s)(x)|^{2}\big)^{p/2}dxdsdt. \label{middle_of_Besov_bound}
\end{eqnarray}

\noindent
By \eqref{bound_of_q_j_3}, the first two integrals on the right hand side of \eqref{middle_of_Besov_bound} are bounded by
\begin{equation}\label{middle_of_Besov_bound_2}
\begin{aligned}
C\int_{0}^{T}\int_{0}^{t}(t-s)^{-1+\delta p}\|g_{0}(s,\cdot)\|^{p}_{L_{p}}dtds
\leq C(T)\int_{0}^{T}\|g_{0}(t,\cdot)\|^{p}_{L_{p}}dt.
\end{aligned}
\end{equation}
By Minkowski's inequality and Fubini's theorem, the third integral is bounded by
\begin{equation*}
\begin{aligned}
C \int_{0}^{T}\int_{s}^{T}\big(\sum_{j=1}^{\infty}|K_{j}(t-s)|^{2} \|g_{j}(s,\cdot)\|_{L_{p}}^{2}\big)^{p/2}dtds,
\end{aligned}
\end{equation*}
where $K_{j}(t-s)=(2^{\frac{2\delta}{\alpha}j+\varepsilon j}(t-s)^{-\frac{1}{p}+\delta}\wedge (t-s)^{-\frac{1}{p}-\frac{\alpha}{2}\varepsilon})$.

If $p=2$ then
\begin{eqnarray*}
&&\sum_{j=1}^{\infty} \int_{0}^{T}\int_{s}^{T}|K_{j}(t-s)|^{2}\|g_{j}(s,\cdot)\|_{L_{2}}^{2}dtds
\\
&\leq& C \int_{0}^{T}\sum_{j=1}^{\infty}\int_{s}^{s+2^{-\frac{2}{\alpha}j}}2^{\frac{4\delta j}{\alpha}+2\varepsilon j}(t-s)^{-1+2\delta}\|g_{j}(s,\cdot)\|_{L_{2}}^{2}dtds
\\
&&+ C \int_{0}^{T}\sum_{j=1}^{\infty}\int_{s+2^{-\frac{2}{\alpha}j}}^{\infty}(t-s)^{-1-\alpha\varepsilon}\|g_{j}(s,\cdot)\|_{L_{2}}^{2}dtds
\\
&=& C \int_{0}^{T}\sum_{j=1}^{\infty}2^{2\varepsilon j}\|g_{j}(s,\cdot)\|_{L_{2}}^{2}ds.
\end{eqnarray*}
This proves the theorem if  $p=2$. 

 If $p>2$, then
\begin{eqnarray*}
&&\int_{0}^{T}\int_{s}^{T}\big(\sum_{j=1}^{\infty}|K_{j}(t-s)|^{2} \|g_{j}(s,\cdot)\|_{L_{p}}^{2}\big)^{p/2}dtds
\\
&\leq& C \int_{0}^{T}\int_{s}^{T}\big(\sum_{j=1}^{\infty}1_{J}(t,s,j)|K_{j}(t-s)|^{2} \|g_{j}(s,\cdot)\|_{L_{p}}^{2}\big)^{p/2}dtds
\\
&&+ \int_{0}^{T}\int_{s}^{T}\big(\sum_{j=1}^{\infty}1_{J^{c}}(t,s,j)|K_{j}(t-s)|^{2} \|g_{j}(s,\cdot)\|_{L_{p}}^{2}\big)^{p/2}dtds,
\end{eqnarray*}
where $J=\{(t,s,j)|2^{j}(t-s)^{\frac{\alpha}{2}}\leq 1\}$.  By \eqref{epsilon_delta}, if $(t,s,j)\in J$, then $K_{j}(t-s)=2^{\frac{2\delta j}{\alpha}+\varepsilon j}(t-s)^{-\frac{1}{p}+\delta}$. Therefore, by H\"older's inequality, we have
\begin{equation*}
\begin{aligned}
&\sum_{j=1}^{\infty}1_{J}|K_{j}(t-s)|^{2} \|g_{j}(s,\cdot)\|_{L_{p}}^{2}
\\
&=\sum_{j=1}^{\infty}1_{J}2^{aj}2^{-aj}2^{\frac{4\delta j}{\alpha}+2\varepsilon j}(t-s)^{-\frac{2}{p}+2\delta}\|g_{j}(s,\cdot)\|_{L_{p}}^{2}
\\
&\leq (t-s)^{-\frac{2}{p}+2\delta}\big(\sum_{j \in J(t,s)}2^{aqj}\big)^{1/q}\big(\sum_{j \in J(t,s)}2^{-\frac{apj}{2}}2^{\frac{2\delta pj}{\alpha}+p\varepsilon j}\|g_{j}(s,\cdot)\|_{L_{p}}^{p}\big)^{2/p},
\end{aligned}
\end{equation*}
where $q=\frac{p}{p-2}$, $a\in(0,\frac{4\delta}{\alpha})$, and $J(t,s)=\{j=1,2,\dots|(t,s,j)\in J \}$. Note that
\begin{equation*}
\big(\sum_{j \in J(t,s)}2^{aqj}\big)^{1/q}\leq C(p) (t-s)^{-\frac{\alpha a}{2}}.
\end{equation*}
Thus we get
\begin{eqnarray}
&&\int_{0}^{T} \int_{s}^{T}\big(\sum_{j=1}^{\infty}1_{J}(t,s,j)|K_{j}(t-s)|^{2} \|g_{j}(s,\cdot)\|_{L_{p}}^{2}\big)^{p/2}dtds   \nonumber
\\
&\leq& C \int_{0}^{T}\sum_{j=1}^{\infty}
\int_{s}^{s+2^{-\frac{2j}{\alpha}}}(t-s)^{-1+p\delta-\frac{p\alpha a}{4}}2^{-\frac{apj}{2}}2^{\frac{2\delta pj}{\alpha}+p\varepsilon j}\|g_{j}(s,\cdot)\|_{L_{p}}^{p}dtds  
\nonumber
\\
&\leq& C \int_{0}^{T}\sum_{j=1}^{\infty}2^{p\varepsilon j}\|g_{j}(t,\cdot)\|_{L_{p}}^{p}dt.   \label{eqn 07.18.1}
\end{eqnarray}

\noindent

Next we consider the remaining part:
\begin{equation*}
\begin{aligned}
&\sum_{j=1}^{\infty}1_{J^{c}}|K_{j}(t-s)|^{2} \|g_{j}(s,\cdot)\|_{L_{p}}^{2}
=\sum_{j=1}^{\infty}1_{J^{c}}2^{bj}2^{-bj}(t-s)^{-\frac{2}{p}-\alpha \varepsilon}\|g_{j}(s,\cdot)\|_{L_{p}}^{2}
\\
&\quad \leq (t-s)^{-\frac{2}{p}-\alpha\varepsilon}\big(\sum_{j \notin J(t,s)}2^{bqj}\big)^{1/q}\big(\sum_{j \notin J(t,s)}2^{-\frac{bpj}{2}}\|g_{j}(s,\cdot)\|_{L_{p}}^{p}\big)^{2/p},
\end{aligned}
\end{equation*}
where $q=\frac{p}{p-2}$, and $b\in(-2\varepsilon,0)$.  Note that
\begin{equation*}
\big(\sum_{j \notin J(t,s)}2^{bqj}\big)^{1/q}\leq C(p) (t-s)^{-\frac{\alpha b}{2}}.
\end{equation*}
Therefore it follows that 
\begin{eqnarray}
&& \int_{0}^{T}\int_{s}^{T}\big(\sum_{j=1}^{\infty}1_{J^{c}}(t,s,j)|K_{j}(t-s)|^{2} \|g_{j}(s,\cdot)\|_{L_{p}}^{2}\big)^{p/2}dtds  \nonumber
\\
&\leq& C \int_{0}^{T}\sum_{j=1}^{\infty}\int_{s+2^{-\frac{2j}{\alpha}}}^{\infty}(t-s)^{-1-\frac{\alpha b p}{4}-\frac{\alpha\varepsilon p}{2}}2^{-\frac{bpj}{2}}\|g_{j}(s,\cdot)\|_{L_{p}}^{p}dtds    \nonumber
\\
&\leq& C \int_{0}^{T}\sum_{j=1}^{\infty}2^{p\varepsilon j}\|g_{j}(t,\cdot)\|_{L_{p}}^{p}dt.       \label{middle_of_Besov_bound_4}
\end{eqnarray}
Combining \eqref{middle_of_Besov_bound_2}, \eqref{eqn 07.18.1} and \eqref{middle_of_Besov_bound_4} we get \eqref{middle_of_Besov_bound_5}  for $p>2$. Hence, the  theorem is proved.
\end{proof}

The next part of this section is related to the non-zero initial value problem
$$
\partial^{\alpha}_tu=\Delta u, \quad t>0 \,\, ; \quad  u(0)=u_0, \,\, 1_{\alpha>1}\partial_{t}u(0)=0.
$$
The solution is given in the form of $p(t,\cdot)\ast u_0$, and we study the regularity of this convolution.

\vspace{3mm}

Define
\begin{equation}\label{def_of_p_j}
\begin{aligned}
p_{j}(t,x)&=(\Psi_{j}(\cdot)\ast p(t,\cdot))(x)
=\cF^{-1}(\hat{\Psi}(2^{-j}\cdot)\hat{p}(t,\cdot))(x)
\\
&=2^{jd}\cF^{-1}(\hat{\Psi}(\cdot)\hat{p}(t,2^{j}\cdot))(2^{j}x)
:=2^{jd}\bar{p}_{j}(t,2^{j}x).
\end{aligned}
\end{equation}

\begin{lemma}\label{L_1_bound_of_p_j}
Let $p>1$, $0<\alpha<2$ and $\alpha\neq 1$. Then there exists a constant $C$ depending only on $\alpha,d$ such that
\begin{equation}\label{bound_of_p_j}
\|p_{j}(t,\cdot)\|_{L_{1}}\leq C (2^{-\frac{2j}{\alpha}}t^{-1}\wedge 1), \quad t>0.
\end{equation}
\end{lemma}
\begin{proof}
Let $R(t,x):=|x|^{2}t^{-\alpha}$. Then by \eqref{bound_of_p_1}, and \eqref{bound_of_p_2},
\begin{equation*}
\begin{aligned}
\int_{\R^{d}}|p(t,x)|dx &= \int_{R\geq1}|p(t,x)|dx+\int_{R< 1}|p(t,x)|dx
\\
&\leq C \int_{R\geq1}t^{-\frac{\alpha d}{2}}\exp{\{-c|x|^{\frac{2}{2-\alpha}}t^{-\frac{\alpha}{2-\alpha}}\}}dx
\\
&\quad+C \int_{R< 1} |x|^{-d}(R+R|\log{R}|1_{d=2}+R^{1/2}1_{d=1})dx.
\end{aligned}
\end{equation*}
By using change of variables and the relation
\begin{equation*}
r^{\nu}|\log{r}|\leq C(\nu) \quad 0<r\leq 1,\quad \nu>0
\end{equation*}
we have $\|p(t,\cdot)\|_{L_{1}}\leq C$.  Due to this and the relation $\|p_j(t,\cdot)\|_{L_1}= \|\bar{p}_j(t,\cdot)\|_{L_1}$,  it only  remains to show
$$
\|\bar{p}_{j}(t,\cdot)\|_{L_{1}}\leq C 2^{-\frac{2j}{\alpha}}t^{-1}.
$$
By definition (see \eqref{def_of_p_j})
\begin{equation}
     \label{eqn 5.16.7}
\cF(\bar{p}_j)(t,\xi)=\hat{\Psi}(\cdot)\hat{p}(t,2^{j}\xi).
\end{equation}
Since $q_{\alpha,\alpha}:=D^{\alpha-\alpha}_t p=p$, by \eqref{fourier_fo_q} and \eqref{representation_of_E_a_1}, we have
\begin{eqnarray}
\nonumber
|\cF{\bar{p}_{j}}(t,\xi)|&\leq& C 1_{\frac{1}{2} \leq |\xi|\leq 2}\int_{0}^{1}r^{\alpha-1}\exp{(-2^{\frac{2j}{\alpha}}|\xi|^{\frac{2}{\alpha}}tr)}rdr
\\
&&+C 1_{\frac{1}{2} \leq |\xi|\leq 2} \int_{1}^{\infty} r^{-\alpha-1} \exp{(-2^{\frac{2j}{\alpha}}|\xi|^{\frac{2}{\alpha}}tr)}rdr.   \label{eqn 5.16.8}
\end{eqnarray}
 Note that for any polynomial $Q$ of degree $m$ and  constant $c>0$, we have
$$
Q(r)e^{-cr}\leq C(c,m)r^{-1}.
$$

\noindent
This and \eqref{eqn 5.16.8} easily yield
\begin{equation*}
|\cF\bar{p}_{j}(t,\xi)|\leq C 2^{-\frac{2j}{\alpha}}t^{-1}1_{\frac{1}{2}\leq |\xi|\leq 2}.
\end{equation*}
Similarly, using \eqref{eqn 5.16.7}  and following above computations,  for any multi-index $\gamma$ we get
\begin{equation*}
|D_{\xi}^{\gamma}\cF\bar{p}_{j}(t,\xi)|\leq C(\alpha,\gamma,d)2^{-\frac{2j}{\alpha}}t^{-1}1_{\frac{1}{2}\leq |\xi|\leq 2}.
\end{equation*}
Therefore, we finally have
\begin{equation*}
\begin{aligned}
\|\bar{p}_{j}(t,\cdot)\|_{L_{1}}&=\int_{\R^{d}}(1+|x|^{2d})^{-1}(1+|x|^{2d})|\bar{p}_{j}(t,x)|dx
\\
&\leq C \int_{\R^{d}}(1+|x|^{2d})^{-1}\sup_{\xi}|(1+\Delta_{\xi}^{d})\cF(\bar{p}_{j})(t,\xi)|dx
\\
&\leq C 2^{-\frac{2j}{\alpha} j}t^{-1}.
\end{aligned}
\end{equation*}
 The lemma is proved.
\end{proof}

\begin{theorem}\label{Besov_bound_2}
Let, $p>1$,  $0<\alpha<2$ and $f\in\Ccinf(\R^{d})$. Then we have
\begin{equation}\label{Besov_bound_of_p}
\int_{0}^{T}\int_{\R^{d}}|p(t,\cdot)\ast f |^{p} dx dt \leq C \|f\|^{p}_{B^{-\frac{2}{\alpha p}}_{p}},
\end{equation}
where the constant $C$ depends only on $\alpha,d,p$, and $T$.
\end{theorem}
\begin{proof}
Since the case   $\alpha=1$ is  a classical result,  we assume $\alpha\neq 1$.  By \eqref{middle_of_Besov_bound_6}, and the relation $\cF(f_{1}\ast f_{2})=\cF(f_{1})\cF(f_{2})$,
\begin{equation*}
\begin{aligned}
\int_{0}^{T}\int_{\R^{d}}|p(t,\cdot)\ast f|^{p} dx dt
&\leq C \int_{0}^{T}(\|p_{0}(t,\cdot)\|_{L_{1}}+\|p_{1}(t,\cdot)\|_{L_{1}})^{p}\|f_{0}\|^{p}_{L_{p}}dt
\\
&\quad+ C \int_{0}^{T}\big(\sum_{j=1}^{\infty}\sum_{i=j-1}^{j+1}\|p_{i}(t,\cdot)\|_{L_{1}}\|f_{j}\|_{L_{p}}\big)^{p}dt.
\end{aligned}
\end{equation*}

\noindent
By \eqref{bound_of_p_j},
\begin{equation}\label{eqn 07.09.1}
\int_{0}^{T}(\|p_{0}(t,\cdot)\|_{L_{1}}+\|p_{1}(t,\cdot)\|_{L_{1}})^{p}\|f_{0}\|^{p}_{L_{p}}dt \leq C(T) \|f_{0}\|_{L_{p}}^{p},
\end{equation}
and
\begin{equation*}
\int_{0}^{T}\big(\sum_{j=1}^{\infty}\sum_{i=j-1}^{j+1}\|p_{i}(t,\cdot)\|_{L_{1}}\|f_{j}\|_{L_{p}}\big)^{p}dt
\leq C \int_{0}^{T}\big(\sum_{j=1}^{\infty}(2^{-\frac{2j}{\alpha}}t^{-1}\wedge 1)\|f_{j}\|_{L_{p}}\big)^{p}dt.
\end{equation*}
Observe that
\begin{equation*}
\begin{aligned}
\int_{0}^{T}&\big(\sum_{j=1}^{\infty}(2^{-\frac{2j}{\alpha}}t^{-1}\wedge 1)\|f_{j}\|_{L_{p}}\big)^{p}dt
\\
&\leq \int_{0}^{T}\big(\sum_{j=1}^{\infty}1_{J}(t,j)\|f_{j}\|_{L_{p}}\big)^{p}dt +\int_{0}^{T}\big(\sum_{j=1}^{\infty}1_{J^{c}}(t,j)2^{-\frac{2j}{\alpha}}t^{-1}\|f_{j}\|_{L_{p}}\big)^{p}dt,
\end{aligned}
\end{equation*}
where $J=\{(t,j)|2^{\frac{2j}{\alpha}}t\leq 1 \}$. By H\"older's inequality, 
\begin{equation*}
\begin{aligned}
\int_{0}^{T}\big(\sum_{j=1}^{\infty}1_{J}\|f_{j}\|_{L_{p}}\big)^{p}dt
&=\int_{0}^{T}\big(\sum_{j\in J(t)}2^{-\frac{2j}{\alpha}a}2^{\frac{2j}{\alpha}a}\|f_{j}\|_{L_{p}}\big)^{p}dt
\\
&\leq \int_{0}^{T}\big(\sum_{j\in J(t)}2^{-\frac{2j}{\alpha}aq}\big)^{p/q}\big(\sum_{j\in J(t)}2^{\frac{2j}{\alpha}ap}\|f_{j}\|^{p}_{L_{p}}\big)dt,
\end{aligned}
\end{equation*}
where $a\in(-\frac{1}{p},0)$, $q=\frac{p}{p-1}$, and $J(t)=\{j=1,2,\dots|(t,j)\in J\}$. Since
\begin{equation*}
\sum_{j\in J(t)}2^{-\frac{2j}{\alpha}aq} \leq C(q,a) t^{aq},
\end{equation*}
we have
\begin{equation}\label{middle_of_Besov_bound_2_2}
\begin{aligned}
\int_{0}^{T}\big(\sum_{j=1}^{\infty}1_{J}\|f_{j}\|_{L_{p}}\big)^{p}dt
&\leq C\sum_{j=1}^{\infty} \int_{0}^{2^{-\frac{2j}{\alpha}}}t^{ap}2^{\frac{2j}{\alpha}ap}\|f_{j}\|^{p}_{L_{p}}dt
\\
&\leq C \sum_{j=1}^{\infty}2^{-\frac{2j}{\alpha}}\|f_{j}\|^{p}_{L_{p}}.
\end{aligned}
\end{equation}
By H\"older's inequality again, for $b\in(-1,-\frac{1}{p})$ and $q=\frac{p}{p-1}$,
\begin{equation*}
\begin{aligned}
\int_{0}^{T}&\big(\sum_{j=1}^{\infty}1_{J^{c}}2^{-\frac{2j}{\alpha}}t^{-1}\|f_{j}\|_{L_{p}}\big)^{p}dt
\\
&=\int_{0}^{T}\big(\sum_{j\notin J(t)}2^{-\frac{2j}{\alpha}b}2^{\frac{2j}{\alpha}b}2^{-\frac{2j}{\alpha}}t^{-1}\|f_{j}\|_{L_{p}}\big)^{p}dt
\\
&\leq \int_{0}^{T}t^{-p}\big(\sum_{j\notin J(t)}2^{-\frac{2j}{\alpha}bq}2^{-\frac{2j}{\alpha}q}\big)^{p/q}\big(\sum_{j\notin J(t)}{2^{\frac{2j}{\alpha}bp}}\|f_{j}\|^{p}_{L_{p}}\big)dt.
\end{aligned}
\end{equation*}
Since
\begin{equation*}
\sum_{j\notin J(t)}2^{-\frac{2j}{\alpha}(b+1)q} \leq C(q,b) t^{(b+1)q},
\end{equation*}
we have
\begin{equation}\label{middle_of_Besov_bound_2_3}
\begin{aligned}
\int_{0}^{T}\big(\sum_{j=1}^{\infty}1_{J^{c}}2^{-\frac{2j}{\alpha}}t^{-1}\|f_{j}\|_{L_{p}}\big)^{p}dt
&\leq C \sum_{j=1}^{\infty}\int_{2^{-\frac{2j}{\alpha}}}^{\infty}t^{-p}t^{(b+1)p}{2^{\frac{2j}{\alpha}bp}}\|f_{j}\|^{p}_{L_{p}}dt
\\
&=C \sum_{j=1}^{\infty}2^{-\frac{2j}{\alpha}}\|f_{j}\|^{p}_{L_{p}}.
\end{aligned}
\end{equation}
Combining \eqref{eqn 07.09.1}, \eqref{middle_of_Besov_bound_2_2} and \eqref{middle_of_Besov_bound_2_3}, we have \eqref{Besov_bound_of_p}. The theorem is proved.
\end{proof}

\vspace{3mm}

The last part of this section is related to the non-zero initial  date problem of the type
\begin{equation}\label{equation_with_nonzero_initial_data_derivative}
\partial_{t}^{\alpha}u=\Delta u, \quad t>0\,\,; \quad u(0,x)=0, \,\, 1_{\alpha>1}\partial_{t}u(0,x)=1_{\alpha>1}v_{0}(x).
\end{equation}

\noindent
Let $\alpha>1$. Then using Lemma \ref{properties_of_p_q}, each $x\neq 0$, one can check that
$$
P(t,x):=q_{\alpha,\alpha-1}=\int_{0}^{t}p(s,x)ds
$$  
is well defined and becomes a fundamental solution to \eqref{equation_with_nonzero_initial_data_derivative}.

For  $j=0,1,2,\dots$ define
\begin{equation}\label{def_of_P_j}
\begin{aligned}
P_j(t,x)&=(\Psi_{j}(\cdot)\ast P(t,\cdot))(x)
\\
&=\cF^{-1}(\hat{\Psi}(2^{-j}\cdot)\hat{P}(t,\cdot))(x)
\\
&=2^{jd}\cF^{-1}(\hat{\Psi}(\cdot)\hat{P}(t,2^{j}\cdot))(2^{j}x)
\\
&:=2^{jd}\bar{P}_{j}(t,2^{j}x).
\end{aligned}
\end{equation}

\begin{lemma}\label{bound_of_P_j}
Let   $\alpha\in (1,2)$. Then, for any $\delta \in (0,\alpha)$,  there exists a constant $C$ depending only on $\alpha,d,\varepsilon,\delta$ such that for any $t>0$,
\begin{equation}\label{bound_of_frac_of_P_j_2}
\|P_{j}(t,\cdot)\|_{L_{1}
}\leq C (2^{-2j+\frac{2\delta}{\alpha}j}t^{1-\alpha+\delta}\wedge t).
\end{equation}
\end{lemma}
\begin{proof}
By \eqref{bounds_of_q_1}, \eqref{bounds_of_q_2} and \eqref{scaling_of_q_2},  we  easily get
$$
\|P(t,\cdot)\|_{L_{1}}\leq C t.
$$
Therefore, it suffices to show that
$$
\|\bar{P}_{j}(t,\cdot)\|_{L_{1}}\leq C 2^{-2j+\frac{2\delta}{\alpha}j}t^{1-\alpha+\delta}.
$$
By definition, 
\begin{equation}
    \label{eqn 5.16.10}
\cF(\bar{P}_j)(t,\xi)=\hat{\Psi}(\xi)\cF(P)(t,2^{j}\xi).
\end{equation}
Also, by  Lemma \ref{integral_rep_of_Mittag_Leffler} and Lemma \ref{integral_rep_of_frac_q} with $\mu=0$, $\beta=\alpha-1$, we have
\begin{equation*}
\begin{aligned}
|\cF&\{P(t,\cdot)\}(2^j\xi)|
\\
&\leq C |2^j\xi|^{-\frac{2}{\alpha}}\int_{0}^{\infty}\frac{\exp{(-m_{1}t|2^j\xi|^{\frac{2}{\alpha}}r )}(|r^{\alpha}\sin{(\psi+m_{3})}|+|\sin{(\psi+m_{4})}|)}{r^{2\alpha}-2r^{\alpha}m_{5}+1}r^{\alpha-2}dr,
\end{aligned}
\end{equation*}
where $\psi =m_{2}t|2^j\xi|^{\frac{2}{\alpha}}r$. Note that for any polynomial $Q$ of degree $m$, and $c>0$, we have
\begin{equation}\label{middle_of_bound_of_P_j}
Q(r)\exp{(-cr)}\leq C(c,m)r^{1-\alpha+\delta},\quad r>0.
\end{equation}
This with the condition $\delta\in (0,\alpha)$  gives
\begin{eqnarray*}
&&|\cF(\bar{P}_j)(t,\xi)|\leq 1_{1/2\leq |\xi|\leq 2}
|\cF(P)(t,2^j\xi)| \\
&\leq& C1_{1/2\leq |\xi|\leq 2} 2^{-\frac{2j}{\alpha}} \big( \int_{0}^{1}(t|2^j\xi|^{\frac{2}{\alpha}}r)^{1-\alpha+\delta}r^{\alpha-2}dr+\int_{1}^{\infty}(t|2^j\xi|^{\frac{2}{\alpha}}r)^{1-\alpha+\delta}r^{-2}dr\big)\\
&\leq&
C2^{-2j+\frac{2\delta}{\alpha}j}t^{1-\alpha+\delta}1_{\frac{1}{2}\leq |\xi|\leq 2}.
\end{eqnarray*}
Using \eqref{eqn 5.16.10} and similar computations above, we also get  for any multi-index $\gamma$
\begin{equation*}
|D_{\xi}^{\gamma}\cF(\bar{P}_{j})|\leq C 2^{-2j+\frac{2\delta}{\alpha}j}t^{1-\alpha
+\delta}1_{\frac{1}{2}\leq |\xi|\leq 2}.
\end{equation*}
Therefore, we have 
\begin{equation*}
\begin{aligned}
\|\bar{P}_{j}(t,\cdot)\|_{L_{1}}&=\int_{\R^{d}}(1+|x|^{2d})^{-1}(1+|x|^{2d})|\bar{P}_{j}(t,x)|dx
\\
&\leq C \int_{\R^{d}}(1+|x|^{2d})^{-1}\sup_{\xi}|(1+\Delta_{\xi}^{d})\cF(\bar{P}_{j})(t,\xi)|dx
\\
&\leq C 2^{-2j+\frac{2\delta}{\alpha}j}t^{1-\alpha+\delta}.
\end{aligned}
\end{equation*}
The lemma is proved.
\end{proof}

\begin{theorem}\label{Besov_bound_3}
Let $\alpha\in (1,2)$ and $h\in \Ccinf(\R^{d})$. Then there exists a constant $C=C(\alpha,d,p,T)$ such that 
\begin{equation}\label{Besov_bound_3_2}
\int_{0}^{T}\int_{\R^{d}}|(P(t)\ast f)(x)|^{p} dx dt \leq C \|h\|^{p}_{B^{-\frac{2}{\alpha p}-\frac{2}{\alpha}}_{p}}, \quad \text{if}\quad \alpha>1+\frac{1}{p}
\end{equation}
 and 
\begin{equation}\label{eqn 05.29.20:18}
\int_{0}^{T}\int_{\R^{d}}|(P(t)\ast f)(x)|^{p} dx dt \leq C  \|h\|^{p}_{B^{-\frac{2}{\alpha}}_{p}}, \quad \text{if} \quad 1<\alpha\leq 1+1/p.
\end{equation}

\end{theorem}

\begin{proof}

\textbf{Case 1}. Let $\alpha>1+1/p$.  Then by assumption on $\alpha$, we can take $\delta \in (0,\alpha)$ such that 
\begin{equation}\label{epsilon'_delta'}
 \alpha-1-\delta-\frac{1}{p}>0,\quad -2+\frac{2\delta}{\alpha}<0.
\end{equation}

By \eqref{middle_of_Besov_bound_6} and the relation $\cF(h_{1}\ast h_{2})=\cF(h_{1})\cF(h_{2})$, 
\begin{equation*}
\begin{aligned}
\int_{0}^{T}&\int_{\R^{d}}|(P(t)\ast h)(x)|^{p} dx dt
\\
&\leq C \int_{0}^{T}(\|P_{0}(t,\cdot)\|_{L_{1}}+\|P_{1}(t,\cdot)\|_{L_{1}})^{p}\|h_{0}\|^{p}_{L_{p}}dt
\\
&\quad+ C \int_{0}^{T}\big(\sum_{j=1}^{\infty}\sum_{i=j-1}^{j+1}\|P_{i}(t,\cdot)\|_{L_{1}}\|h_{j}\|_{L_{p}}\big)^{p}dt.
\end{aligned}
\end{equation*}

\noindent
 Note that \eqref{bound_of_frac_of_P_j_2} with  \eqref{epsilon'_delta'} easily yields
 \begin{equation}\label{middle_of_Besov_bound_3}
\begin{aligned}
\int_{0}^{T}(\|P_{0}(t,\cdot)\|_{L_{1}}+\|P_{1}(t,\cdot)\|_{L_{1}})^{p}\|h_{0}\|^{p}_{L_{p}}dt
\leq C(T) \|h_{0}\|^{p}_{L_{p}}.
\end{aligned}
\end{equation}
Also  by \eqref{epsilon'_delta'},
\begin{equation*}
\begin{aligned}
&\int_{0}^{T}\big(\sum_{j=1}^{\infty}\sum_{i=j-1}^{j+1}\|P_{i}(t,\cdot)\|_{L_{1}}\|h_{j}\|_{L_{p}}\big)^{p}dt
\\
&\quad \leq C \int_{0}^{T}\big(\sum_{j=1}^{\infty}L_{j}(t)\|h_{j}\|_{L_{p}}\big)^{p}dt
\\
&\quad \leq C \int_{0}^{T}\big(\sum_{j=1}^{\infty}1_{J}(t,j)L_{j}(t)\|h_{j}\|_{L_{p}}\big)^{p}dt
+C\int_{0}^{T}\big(\sum_{j=1}^{\infty}1_{J^{c}}(t,j)L_{j}(t)\|h_{j}\|_{L_{p}}\big)^{p}dt,
\end{aligned}
\end{equation*}
where $J:=\{(t,j)|2^{j}t^{\frac{\alpha}{2}}\geq 1\}$,  and 
$$
L_{j}(t):=(2^{-2j+\frac{2\delta}{\alpha}j}t^{1-\alpha+\delta}\wedge t) =
\begin{cases}
2^{-2j+\frac{2\delta}{\alpha}j}t^{1-\alpha+\delta}  \quad &: (t,j)\in J \\
t \quad &: (t,j)\notin J.
\end{cases}
$$
By H\"older's inequality,
\begin{equation}\label{middle_of_Besov_3_4}
\begin{aligned}
\int_{0}^{T}\big(&\sum_{j=1}^{\infty}1_{J}L_{j}(t)\|h_{j}\|_{L_{p}}\big)^{p}dt
\\
&=\int_{0}^{T}\big(\sum_{j=1}^{\infty}1_{J}2^{-2j+\frac{2\delta}{\alpha}j}t^{1-\alpha-\delta}2^{-\frac{2bj}{\alpha}}2^{\frac{2bj}{\alpha}}\|h_{j}\|_{L_{p}}\big)^{p}dt
\\
&\leq \int_{0}^{T}t^{(1-\alpha+\delta)p}\big(\sum_{j\in J(t)}2^{-\frac{2bqj}{\alpha}}\big)^{p/q}\big(\sum_{j\in J(t)}2^{-2pj+\frac{2\delta}{\alpha}pj}2^{\frac{2pbj}{\alpha}}\|h_{j}\|^{p}_{L_{p}}\big)dt,
\end{aligned}
\end{equation}
where $b\in(0,\alpha-1-\frac{1}{p}-\delta)$, $q=\frac{p}{p-1}$, and $J(t)=\{j=1,2,\dots|(t,j)\in J\}$. Since
\begin{equation*}
\big(\sum_{j\in J(t)}2^{-\frac{2qb}{\alpha}j}\big)^{p/q}\leq C(\alpha,p)t^{bp}, 
\end{equation*}
we have
\begin{equation}\label{middle_of_Besov_bound_3_2}
\begin{aligned}
\int_{0}^{T}\big(&\sum_{j=1}^{\infty}1_{J}L_{j}(t)\|h_{j}\|_{L_{p}}\big)^{p}dt
\\
&\leq C \int_{0}^{T}t^{(1-\alpha+\delta)p}t^{bp}\big(\sum_{j\in J(t)}2^{-2pj+\frac{2\delta}{\alpha}pj}2^{\frac{2pbj}{\alpha}}\|h_{j}\|^{p}_{L_{p}}\big)dt
\\
&\leq C\sum_{j=1}^{\infty} \int_{2^{-\frac{2j}{\alpha}}}^{\infty}t^{(1-\alpha+\delta)p}t^{bp}2^{-2pj+\frac{2\delta}{\alpha}pj}2^{\frac{2pbj}{\alpha}}\|h_{j}\|^{p}_{L_{p}}
\\
&\leq C \sum_{j=1}^{\infty}2^{-\frac{2pj}{\alpha }-\frac{2j}{\alpha}}\|h_{j}\|^{p}_{L_{p}}.
\end{aligned}
\end{equation}
Again by H\"older's inequality,
\begin{equation*}
\begin{aligned}
\int_{0}^{T}\big(\sum_{j=1}^{\infty}1_{J^{c}}L_{j}(t)\|h_{j}\|_{L_{p}}\big)^{p}dt
&=\int_{0}^{T}\big(\sum_{j\notin J(t)} t2^{-\frac{2aj}{\alpha}}2^{\frac{2aj}{\alpha}}\|h_{j}\|_{L_{p}}\big)^{p}dt
\\
&\leq \int_{0}^{T}t^{p}\big(\sum_{j\notin J(t)}2^{-\frac{2aqj}{\alpha}}\big)^{p/q}\big(\sum_{j\notin J(t)}2^{\frac{2paj}{\alpha}}\|h_{j}\|^{p}_{L_{p}}\big)dt,
\end{aligned}
\end{equation*}
where $a\in(-1-\frac{1}{p},0)$, and $q=\frac{p}{p-1}$.
Since
\begin{equation*}
\big(\sum_{j\notin J(t)}2^{-\frac{2aqj}{\alpha}}\big)^{p/q} \leq C(\alpha,p) t^{ap},
\end{equation*}
 we have
\begin{equation}\label{middle_of_Besov_bound_3_3}
\begin{aligned}
\int_{0}^{T}\big(\sum_{j=1}^{\infty}1_{J^{c}}L_{j}(t)\|h_{j}\|_{L_{p}}\big)^{p}dt
&\leq C \int_{0}^{T}t^{p+ap}\big(\sum_{j\notin J(t)}2^{\frac{2paj}{\alpha}}\|h_{j}\|^{p}_{L_{p}}\big)dt
\\
&\leq C\sum_{j=1}^{\infty} \int_{0}^{2^{-\frac{2j}{\alpha}}}t^{p+ap}2^{\frac{2paj}{\alpha}}\|h_{j}\|^{p}_{L_{p}}dt
\\
&\leq C \sum_{j=1}^{\infty}2^{-\frac{2pj}{\alpha }-\frac{2j}{\alpha}}\|h_{j}\|^{p}_{L_{p}}.
\end{aligned}
\end{equation}
Combining \eqref{middle_of_Besov_bound_3}, \eqref{middle_of_Besov_bound_3_2}, and \eqref{middle_of_Besov_bound_3_3}, we get \eqref{Besov_bound_3_2}. The theorem is proved.

\vspace{3mm}

\textbf{Case 2}.  Let $1\leq \alpha<1+1/p$. This time, we choose $\delta, b>0$ such that
\begin{equation}
\alpha-1-\delta>0, \quad  b\in (0,\alpha-1-\delta),
\end{equation}
and repeat the proof of Case 1. The only difference is we need to replace \eqref{middle_of_Besov_bound_3_2} by the following:
\begin{equation*}
\begin{aligned}
\int_{0}^{T}\big(&\sum_{j=1}^{\infty}1_{J}L_{j}(t)\|h_{j}\|_{L_{p}}\big)^{p}dt
\\
&\leq C \int_{0}^{T}t^{(1-\alpha+\delta)p}t^{bp}\big(\sum_{j\in J(t)}2^{-2pj+\frac{2\delta}{\alpha}pj}2^{\frac{2pbj}{\alpha}}\|h_{j}\|^{p}_{L_{p}}\big)dt\\
&\leq C \int_{0}^{T}t^{(1-\alpha+\delta-1/p)p}t^{bp}\big(\sum_{j\in J(t)}2^{-2pj+\frac{2\delta}{\alpha}pj}2^{\frac{2pbj}{\alpha}}\|h_{j}\|^{p}_{L_{p}}\big)dt\\
&\leq C\sum_{j=1}^{\infty} \int_{2^{-\frac{2j}{\alpha}}}^{\infty}t^{(1-\alpha+\delta-1/p)p}t^{bp}2^{-2pj+\frac{2\delta}{\alpha}pj}2^{\frac{2pbj}{\alpha}}\|h_{j}\|^{p}_{L_{p}}
\\
&\leq C \sum_{j=1}^{\infty}2^{-\frac{2pj}{\alpha }}\|h_{j}\|^{p}_{L_{p}}.
\end{aligned}
\end{equation*}
On the other hand,  \eqref{middle_of_Besov_bound_3_3}  still holds without any changes, and this certainly implies
$$
\int_{0}^{T}\big(\sum_{j=1}^{\infty}1_{J^{c}}L_{j}(t)\|h_{j}\|_{L_{p}}\big)^{p}dt \leq C
 \sum_{j=1}^{\infty}2^{-\frac{2pj}{\alpha }}\|h_{j}\|^{p}_{L_{p}}. 
 $$
 Hence, Case 2 is also proved.
\end{proof}

\mysection{Proof of  Theorem \ref{thm 10.08.10:57}}

We first prove a version of Theorem \ref{thm 10.08.10:57} for  (deterministic) equation \eqref{eqn deterministic}.

\begin{lemma}\label{deterministic}
Let $0<\alpha<2$, $1<p<\infty$ and $\gamma\in\R$. Then for any $u_{0}\in U^{\gamma+2}_{p}$, $v_{0}\in V^{\gamma+2}_{p}$
and $f\in\bH^{\gamma}_{p}(T)$ the equation
\begin{equation}\label{eqn deterministic}
\partial^{\alpha}_{t}u=\Delta u+f,\quad t>0,x\in\R^{d}\,; \quad u(0)=u_{0},\,\,1_{\alpha>1}\partial_{t}u(0)=1_{\alpha>1}v_0
\end{equation}
has a unique solution $u\in\cH^{\gamma+2}_{p}(T)$, and moreover
\begin{equation}\label{eqn 3.2}
 \|u\|_{\mathcal{H}^{\gamma+2}_{p}(T)}     \leq C 
 \big(\|u_{0}\|_{U^{\gamma+2}_{p}}+1_{\alpha>1}\|v_{0}\|_{V^{\gamma+2}_{p}}+\|f\|_{\bH^{\gamma}_{p}(T)}\big),
\end{equation}
where the constant $C$ depends only on $\alpha,d,p,\gamma$, and $T$.
\end{lemma}

\begin{proof}

Due to Remark \ref{fractional_laplacian_besov}, it is enough to prove the lemma for a particular $\gamma$, and therefore we assume  $\gamma=-2$.  

 The statements of the lemma hold if  $u_0=v_0=0$ due to 
   \cite[Theorem 2.3]{kim16timefractionalspde}
(or  \cite[Theorem 2.10]{kim17timefractionalpde}),  from which   the uniqueness result  follows. Furthermore, considering $u-v$, where $v$ is the solution to the equation with $u_0=v_0=0$ taken from   \cite[Theorem 2.3]{kim16timefractionalspde}, we may assume that $f=0$. By  $\bL_c$ we denote the space of the functions $g$ of the form
\begin{equation*}
g(\omega,x)=\sum_{i=1}^{n}1_{A_{i}}(\omega)g_{i}(x), \quad A_{i}\in\mathscr{F}_{0}, \,g_{i}\in \Ccinf(\R^{d}).
\end{equation*}
Then one can easily check that $\mathbb{L}_{c}$ is dense in $L_{p}(\Omega,\mathscr{F}_{0};H^{\gamma}_{p})$ for any $\gamma\in\bR$. Now, let $u_0,v_0 \in \bL_c$ and define
$$
u(t,x):=(p(t,\cdot)\ast u_0(\cdot))(x)+ 1_{\alpha>1} (P(t,\cdot)\ast v_0(\cdot))(x).
$$
Then by Lemma \ref{properties_of_p_q} (or see \cite[Lemma 3.5]{kim17timefractionalpde} for more detail), $u$ satisfies equation \eqref{eqn deterministic}, and $u\in \bH^n_p(T)$ for any $n\in \bR$, since $u_0,v_0\in \bL_c$. Moreover, for this solution we have
$$
\|u\|_{\mathbb{H}^{\gamma+2}_{p}(T)} \leq C 
 \big(\|u_{0}\|_{U^{\gamma+2}_{p}}+1_{\alpha>1}\|v_{0}\|_{V^{\gamma+2}_{p}}+\|f\|_{\bH^{\gamma}_{p}(T)}\big)
$$
 with $\gamma=-2$ due to Theorem \ref{Besov_bound_2}, Theorem \ref{Besov_bound_3}, and Remark \ref{fractional_laplacian_besov}.  This estimate and  the definition of norm in $\mathcal{H}^{0}_{p}(T)$ certainly yield \eqref{eqn 3.2}.  

In general, take $u^n_0, v^n_0 \in \bL_c$ such that $u^n_0 \to u_{0}$ in $U^{0}_p$ and $v^n_0 \to v_0$ in $V^{0}_p$, and for each $n$ let $u_n$ denote the solution to the equation with initial data $u^n_0$ and $v^n_0$. Then  estimate \eqref{eqn 3.2} corresponding to $u_n-u_m$, where $n,m\in \bN$, shows that 
$u_n$ is a Cauchy sequence in $\mathcal{H}^0_p(T)$, which is a Banach space. Now it is easy to check that the limit of the Cauchy sequence becomes a solution to the equation with initial data $u_0$ and $v_0$, and the estimate also follows.  The lemma is proved.
\end{proof}

\begin{remark}
The proof of  Lemma \ref{deterministic} actually shows that the lemma holds for any $p\in (1,\infty)$ with appropriate Besov spaces.  Precisely speaking, if $\alpha>1+1/p$, then we can use $B^{\gamma+2-2/\alpha p}_{p}$ and $B^{\gamma+2-2/\alpha-2/\alpha p}_{p}$ instead of $U^{\gamma+2}_{p}$, and $V^{\gamma+2}_{p}$ respectively, and for $\alpha\leq 1+1/p$, then we can use $B^{\gamma+2-2/\alpha p}_{p}$ and $B^{\gamma+2-2/\alpha}_{p}$ instead of $U^{\gamma+2}_{p}$, and $V^{\gamma+2}_{p}$ respectively.
\end{remark}

 For $l_2$-valued functions $h$, we write  $h\in\bH^{\infty}_{c}(T,l_{2})$ if $h^{k}=0$ for all large $k$, and each $h^{k}$ is of the type
\begin{equation*}
h^{k}(t,x)=\sum_{i=1}^{n}1_{(\tau_{i-1},\tau_{i}]}(t)g^{ik}(x),
\end{equation*}
where $\tau_{i}$ are bounded stopping times, $\tau_i\leq \tau_{i+1}$, and $g^{ik}\in\Ccinf(\R^{d})$. The space $\bH^{\infty}_{c}(T,l_{2},d_1)$ is defined similarly.  By \cite[Theorem 3.10]{kry99analytic}, $\bH^{\infty}_{c}(T,l_{2})$ is dense in $\bH^{\gamma}_{p}(T,l_{2})$.

\begin{lemma}\label{lem 07.22.1}
Let $\alpha \in (0,2), \beta_2<\alpha+1/p$ and $h\in \bH^{\infty}_c(T)$.  Denote
\begin{equation}
\label{eqn 5.19.1}
u(t,x):=\sum_{k=1}^{\infty}\int_{0}^{t}\left(\int_{\R^{d}}q_{\alpha,\beta_{2}}(t-s,x-y)h^{k}(s,y)dy\right) \cdot dZ^{k}_{s}.
\end{equation}
Then $u\in\mathcal{H}^{2}_{p}(T)$ and satisfies
\begin{equation}\label{3.2}
\partial^{\alpha}_{t}u=\Delta u +\partial^{\beta_{2}}_{t}\int_{0}^{t}h^{k}(s,x) \cdot dZ^{k}_{s},\,\, t>0,x\in\R^{d};\quad u(0)=\partial_{t}u(0)1_{\alpha>1}=0
\end{equation}
in the sense of Definition \ref{def 07.11.1}.
\end{lemma}

\begin{proof}
It is enough to repeat the proof of   \cite[Lemma 3.10]{chen2015fractional}, which deals with the equation driven by Brownian motions. 
\end{proof}

Next we prove  a version of Theorem \ref{thm 10.08.10:57} for the linear  equation 
\begin{eqnarray}
&&\partial^{\alpha}_{t}u=\Delta u +f+\partial^{\beta_{1}}_{t}\int_{0}^{t}g^{k}(s,x)dW^{k}_{s}+\partial^{\beta_{2}}_{t}\int_{0}^{t}h^{k}(s,x) \cdot dZ^{k}_{s}, \quad t>0,x\in\R^{d},  \nonumber
\\
&&\quad u(0)=u_{0},\quad \partial_{t}u(0)1_{\alpha>1}=1_{\alpha>1}v_{0}, \label{eqn 07.18.3}
\end{eqnarray}

\begin{theorem}\label{thm 07.24.1}
Let $\gamma\in\R$, $p\geq2$, $\beta_{1}<\alpha+1/2$ and $\beta_{2}<\alpha+1/p$.  Then, for any  $u_{0}\in U^{\gamma+2}_{p}$, $v_{0}\in V^{\gamma+2}_{p}$, $f\in\bH^{\gamma}_{p}(T)$, $g\in\bH^{\gamma+c_{0}}_{p}(T,l_{2})$ and $h\in\bH^{\gamma+\bar{c}_{0}}_{p}(T,l_{2},d_{1})$, equation \eqref{eqn 07.18.3} has a unique solution $u$ in the class $\mathcal{H}^{\gamma+2}_{p}(T)$, and for this solution  it holds that 
\begin{equation}\label{eqn 07.19.1}
\begin{aligned}
  \|u\|_{\mathcal{H}^{\gamma+2}_{p}(T)}      \leq C  \big(\|u_{0}\|_{U^{\gamma+2}_{p}}&+1_{\alpha>1}\|v_{0}\|_{V^{\gamma+2}_{p}}+\|f\|_{\bH^{\gamma}_{p}(T)}
\\
&+\|g\|_{\bH^{\gamma+c_{0}}_{p}(T,l_{2})}+\|h\|_{\bH^{\gamma+\bar{c}_{0}}_{p}(T,l_{2},d_{1})}\big),
\end{aligned}
\end{equation}
where  $C=C(\alpha,\beta_{1},\beta_{2},d,d_{1},p,\gamma,T)$. 

\end{theorem}

\begin{proof}
  Due to Remark \ref{fractional_laplacian_besov} it is enough to prove the lemma for   $\gamma=0$. The uniqueness follows from Lemma \ref{deterministic}.
  
 Recall that the lemma holds if $h=0$ and  $u_0=v_0=0$ by  \cite[Theorem 2.3]{kim16timefractionalspde},  and it holds if $f=0, g=0, h=0$ by Lemma \ref{deterministic}.  By the linearity of the equation, if $h=0$ then   the existence and the desired estimate  is easily obtained by combining   \cite[Theorem 2.3]{kim16timefractionalspde} and Lemma \ref{deterministic}.  The case $h=0$ is proved.

  Furthermore,  by the result for the case $h=0$ and the linearity of the equation,    
 to finish the proof of the lemma, we only need to   prove the existence result and estimate \eqref{eqn 07.19.1}, provided that $u_0=v_0=0, f=0$ and $g=0$. Also it suffices to prove \eqref{eqn 07.19.1} with $\|u\|_{\mathbb{H}^{\gamma+2}_{p}(T)}$ in place of $\|u\|_{\mathcal{H}^{\gamma+2}_{p}(T)}$ due to the definition of $\|u\|_{\mathcal{H}^{\gamma+2}_{p}(T)}$.
 We divide the proof of this into following three cases.

{\bf{Case 1}}. Let $\beta_2>1/p$. 

If $h\in \bH^{\infty}_c(T,l_2,d_1)$, we define  $u\in \cH^2_p(T)$  as in \eqref{eqn 5.19.1} such that it becomes a solution to equation \eqref{3.2}. Denote 
 $c_1:=\frac{2(\alpha+1/p-\beta_2)}{\alpha}$ and take a small  constant $\varepsilon \in (0,c_1)$ satisfying \eqref{epsilon_delta} with $\beta_{2}$ in place of $\beta$, and  set
  \begin{equation*}
v:=(-\Delta)^{(2-c_{1}-\varepsilon)/2}u,\quad \bar{h}:=(-\Delta)^{(2-c_1-\varepsilon)/2}h.
\end{equation*}
By Burkholder-Davis-Gundy inequality, and \eqref{eqn 5.19.5} 
\begin{equation*}
\begin{aligned}
\|\Delta u\|^{p}_{\bL_{p}(T)}&=\|(-\Delta)^{(c_1+\varepsilon)/2}v\|^{p}_{\bL_{p}(T)}
\\
&\leq C \bE \int_{\R^{d}}\int_{0}^{T}\left(\int_{0}^{t}\sum_{k=1}^{\infty}\left|(-\Delta)^{\frac{c_{1}+\varepsilon}{2}}q_{\alpha,\beta_{2}}(t-s,\cdot)\ast\bar{h}^{k}(s,\cdot)\right|^{2}(x)ds\right)^{\frac{p}{2}}dtdx
\\
&\quad +C \bE\int_{\R^{d}}\int_{0}^{T}\int_{0}^{t}\sum_{k=1}^{\infty}\left|(-\Delta)^{\frac{c_{1}+\varepsilon}{2}}q_{\alpha,\beta_{2}}(t-s,\cdot)\ast\bar{h}^{k}(s,\cdot)\right|^{p}(x)dsdtdx.
\end{aligned}
\end{equation*}
By \cite[Theorem 3.1]{kim16timefractionalspde} we have
\begin{equation*}
\begin{aligned}
\bE \int_{\R^{d}}\int_{0}^{T}\left(\int_{0}^{t}\sum_{k=1}^{\infty}\left|(-\Delta)^{\frac{c_{1}+\varepsilon}{2}}q_{\alpha,\beta_{2}}(t-s,\cdot)\ast\bar{h}^{k}(s,\cdot)\right|^{2}(x)ds\right)^{\frac{p}{2}}dtdx \leq C\|\bar{h}\|^{p}_{\bL_{p}(T,l_{2},d_{1})}, 
\end{aligned}
\end{equation*}
where the constant $C$ depends only on $\alpha,\beta_{2},d,d_{1}$, and $p$. Also by Theorem \ref{Besov_bound} and Remark \ref{fractional_laplacian_besov}\begin{equation*}
\begin{aligned}
&\bE\int_{\R^{d}}\int_{0}^{T}\int_{0}^{t}\sum_{k=1}^{\infty}\left|(-\Delta)^{\frac{c_{1}+\varepsilon}{2}}q_{\alpha,\beta_{2}}(t-s,\cdot)\ast \bar{h}^{k}(s,\cdot)\right|^{p}(x)dsdtdx 
\\
&\quad \leq C\bE \sum_{r=1}^{d_{1}}\sum_{k=1}^{\infty}\int_{0}^{T}\|\bar{h}^{rk}(t,\cdot)\|^{p}_{B^{\varepsilon}_{p}}dt \leq C \bE \sum_{r=1}^{d_{1}}\sum_{k=1}^{\infty}\int_{0}^{T}\|h^{rk}(t,\cdot)\|^{p}_{H^{2-c_1}_{p}} dt,
\end{aligned}
\end{equation*}
where the constant $C$ depends only on $\alpha,\beta_{2},d,d_{1}$, and $p$.  The above estimations and  the inequality $\sum_{k=1}^{\infty}|a_{k}|^{p}\leq \left(\sum_{k=1}^{\infty}|a_{k}|^{2}\right)^{p/2}$ yield
\begin{equation}
  \label{eqn 5.19.20}
  \|\Delta u\|^{p}_{\bL_{p}(T)}\leq  C \|h\|^{p}_{\bH^{2-c_1}_{p}(T,l_{2},d_{1})}=C \|h\|^{p}_{\bH^{\bar{c}_{0}}_{p}(T,l_{2},d_{1})}.
\end{equation}
Also, due to \eqref{eqn 07.26.3} and the inequality $\|\cdot\|_{\bL_p(s)}\leq \|\cdot\|_{\bL_p(T)}$ for $s\leq T$, we have
\begin{eqnarray*}
\|u\|^{p}_{\bL_{p}(T)}&\leq& C \int_{0}^{T}(T-s)^{\theta-1}\Big(\|\Delta u\|^{p}_{\bL_{p}(T)}+\|h\|^{p}_{\bL_{p}(T,l_{2},d_{1})} \Big)ds\\
&\leq& C (\|\Delta u\|^p_{\bL_p(T)}+ \|h\|^{p}_{\bL_{p}(T,l_{2},d_{1})}) \leq C  \|h\|^{p}_{\bH^{\bar{c}'_{0}}_{p}(T,l_{2},d_{1})}.
\end{eqnarray*}
This, \eqref{eqn 5.19.20} and the inequality $\|u\|_{H^2_p}\leq \|u\|_{L_p}+\|\Delta u\|_{L_p}$ yield estimate \eqref {eqn 07.19.1}.

For general $h\in \bH^{\bar{c}_0}_p(T,l_2,d_1)$, it is enough to repeat the approximation argument used in the proof of Lemma \ref{deterministic}. 

\vspace{2mm}

{\bf{Case 2}}. Let  $\beta_{2}=1/p$.  
The argument used in Case 1 shows that to prove the existence result and estimate  \eqref{eqn 07.19.1} we may assume $h\in \bH^{\infty}_c(T,l_2,d_1)$. In this case the existence is a consequence of Lemma \ref{lem 07.22.1}. 

Let $u\in \cH^2_p(T)$ be the solution to the equation. Take $\kappa>0$ from \eqref{eqn 08.09.1}, and put $\kappa'=\kappa\alpha/2$ 
and $\beta'_{2}=1/p+\kappa' >1/p$. 
 Then by Case 1 with $\bar{c}'_{0}=(2\beta'_{2}-2/p)/\alpha=\kappa$, if we define $v$ as in \eqref{eqn 5.19.1} with $\beta'_{2}$ in place of $\beta_{2}$, then $v$ satisfies \eqref{3.2} (with $\beta'_{2}$), and 
\begin{equation}\label{eqn 11.09.16:40}
\|v\|_{\mathcal{H}^{2}_{p}(T)} \leq C \|h\|_{\mathbb{H}^{\kappa}_{p}(T,l_{2},d_{1})}.
\end{equation}
Since $I^{\kappa'}_{t}v$ satisfies \eqref{3.2} (with $\beta_{2}$), by the uniqueness of solutions, we obtain $u(t,x)=I^{\kappa'}_{t}v(t,x)$, and \eqref{eqn 07.19.1} holds due to \eqref{Jensen_fractional_integral} and \eqref{eqn 11.09.16:40}. Hence the case  $\beta_2=1/p$ is also proved. 

\vspace{2mm}

{\bf{Case 3}}. Let  $\beta_{2}<1/p$.   As in Case 2, we only need to prove estimate  \eqref{eqn 07.19.1}, provide that  $h\in \bH^{\infty}_c(T,l_2,d_1)$ and the solution $u$ already exists.

Put
$$
\bar{f}(t,x)=:\frac{1}{\Gamma(1-\beta_{2})}\int^t_0(t-s)^{-\beta_2}h^k(x)\cdot dZ^k_s.
$$
Then by the Burkerholder-Davis-Gundy inequality and \eqref{eqn 07.16.1},
\begin{equation}
   \label{eqn 5.20.1}
\|\bar{f}\|^p_{\bL_p(T)}\leq C \bE \int^T_0 \int^t_0 (t-s)^{-\beta_2 p}|h(s,\cdot)|^p_{L_p(l_2,d_1)} ds dt \leq C \|h\|^p_{\bL_p(T,l_2,d_1)}.
\end{equation}
Note that by Lemma  \ref{lem 07.11.1} (iii), $u$ satisfies
$$
\partial^{\alpha}_tu=\Delta u+\bar{f}, \, t>0\,; \quad u(0)=1_{\alpha>1}\partial_{t}u(0)=0.
$$
Therefore, estimate \eqref{eqn 07.19.1} follows from \eqref{eqn 5.20.1} and Lemma \ref{deterministic}. The theorem is proved.
\end{proof}

\vspace{3mm}

\textbf{Proof of Theorem \ref{thm 10.08.10:57}}.

\textbf{1. Linear case}. 
Due to the method of continuity (see e.g. \cite[Lemma 5.1]{kim16timefractionalspde}) and Theorem \ref{thm 07.24.1} we only need to prove that 
 a priori estimate \eqref{eqn 10.22.14:25} holds, provided that a solution $u\in \cH^{\gamma+2}_p(T)$ to equation  \eqref{eqn 07.16.2} already exists.  Also note that due to the definition of the norm in $\mathcal{H}^{\gamma+2}_{p}(T)$, we only need to prove \eqref{eqn 10.22.14:25} with $\|u\|_{\mathbb{H}^{\gamma+2}_{p}(T)}$ in place of $\|u\|_{\mathcal{H}^{\gamma+2}_{p}(T)}$. 
  
Step 1.
  Assume $u_0=v_0=0$.
  Denote
 $$
 \bar{f}:=b^{i}u_{x^{i}}+cu+f,\quad
 \bar{g}^k:=\mu^{ik}u_{x^i}+\nu^ku+g^k, \quad 
 \bar{h}^k:=\bar{\mu}^{ik}u_{x^i}+\bar{\nu}^ku+h^k.
 $$
Recall that $c_0, \bar{c}_0<2$. By Assumption \ref{asm 07.16.1}, $\mu=0$ if $c_0\geq 1$, and $\bar{\mu}=0$ if $\bar{c}_0 \geq 1$. Therefore, by \eqref{eqn 08.06.1}, 
 \begin{eqnarray*}
  \|\bar{g}\|_{\bH^{\gamma+c_0}_p(t,l_2)}&\leq& C 1_{c_0<1}\|u_x\|_{\bH^{\gamma+c_0}_p(T)}+C\|u\|_{\bH^{\gamma+c_0}_p(T)}+\|g\|_{\bH^{\gamma+c_0}_p(T,l_2)} \\
   &\leq&  C 1_{c_0<1}\|u\|_{\bH^{\gamma+c_0+1}_p(T)}+C\|u\|_{\bH^{\gamma+c_0}_p(T)}+\|g\|_{\bH^{\gamma+c_0}_p(T,l_2)}.
   \end{eqnarray*}
   The similar estimate holds for $\bar{f}$ and $\bar{h}$.  Using these and 
   the embedding inequality
\begin{equation}\label{eqn 10.08.13:18}
 \|u\|_{H^{\gamma+\delta}_p}\leq \varepsilon \|u\|_{H^{\gamma+2}_p}+C(\delta,\varepsilon)\|u\|_{H^{\gamma}_p}, \quad \delta\in (0,2), \,\varepsilon>0,
\end{equation}
 we get,   for any $\varepsilon>0$ 
  and $t\leq T$,
  \begin{equation} \label{eqn 5.20.4} 
  \begin{aligned}
  &\|\bar{f}\|_{\bH^{\gamma}_p(t)}+\|\bar{g}\|_{\bH^{\gamma+c_0}_p(t,l_2)} +   \|\bar{h}\|_{\bH^{\gamma+\bar{c}_0}_p(t,l_2,d_1)} \\
 & \quad \leq \varepsilon \|u\|_{\bH^{\gamma+2}_p(t)}
      +C \|u\|_{\bH^{\gamma}_p(t)}
      \\
&\quad\quad +  \|f\|_{\bH^{\gamma}_p(t)} + \|g\|_{\bH^{\gamma+c_0}_p(t,l_2)} 
      +\|h\|_{\bH^{\gamma+\bar{c}_0}_p(t,l_2,d_1)} <\infty.      
\end{aligned}
\end{equation}
  Recall $\frac{\partial}{\partial x^i}: H^{\nu}_p\to H^{\nu-1}_p$ is a bounded operator for any $\nu\in \bR$. Using this, \eqref{eqn 08.06.1} and Assumption \ref{asm 07.16.1},  we easily have
 \begin{eqnarray}
   \nonumber                   
&&\|\bar{f}\|_{\mathbb{H}^{\gamma}_{p}(T)} + \|\bar{g}\|_{\bH^{\gamma+c_0}_p(T,l_2)}+\|\bar{h}\|_{\bH^{\gamma+\bar{c}_0}_p(T,l_2,d_1)} \\
&\leq &
  C \|u\|_{\bH^{\gamma+2}_p(T)}+\|f\|_{\bH^{\gamma}_p(T)}
  +\|g\|_{\bH^{\gamma+c_0}_p(T,l_2)}+\|h\|_{\bH^{\gamma+\bar{c}_0}_p(T,l_2,d_1)}.   
  \label{eqn 5.21.4}  
  \end{eqnarray}

Due to Theorem \ref{thm 07.24.1} and \eqref{eqn 5.20.4},  we  can define $v\in \cH^{\gamma+2}_p(T)$ as the solution  to equation  \eqref{eqn 07.18.3}  with $\bar{g}$ and $\bar{h}$ in place of  $g$ and $h$, respectively. Furthermore,  for each $t\leq T$ we have
\begin{eqnarray}
\nonumber
\|v\|_{\bH^{\gamma+2}_p(t)} &\leq& C\|f\|_{\bH^{\gamma}_p(t)}
+C\|\bar{g}\|_{\bH^{\gamma+c_0}_p(t,l_2)}+C\|\bar{h}\|_{\bH^{\gamma+\bar{c}_0}_p(t,l_2,d_1)}. \label{eqn 5.21.1}
\end{eqnarray}

 \noindent
 Note that $\bar{u}:=u-v\in \cH^{\gamma+2}_p(T)$ satisfies
 $$
 \partial^{\alpha}_t \bar{u}=a^{ij}\bar{u}_{x^ix^j}+\tilde{f} , \,\,t>0\,;\quad \bar{u}(0)=1_{\alpha>1}\partial_{t}\bar{u}_t(0)=0,
 $$
 where
 $$
 \tilde{f}:=(a^{ij}-\delta^{ij})v_{x^ix^j}+\bar{f}-f.
 $$
 Therefore, by \cite[Theorem 2.10]{kim17timefractionalpde} and \eqref{eqn 5.20.4}, for each $t\leq T$
 \begin{eqnarray*}
 \|u\|_{\bH^{\gamma+2}_p(t)}&\leq& \|u-v\|_{\bH^{\gamma+2}_p(t)}+\|v\|_{\bH^{\gamma+2}_p(t)}\\
 &\leq&C \varepsilon \|u\|_{\bH^{\gamma+2}_p(t)}+C \|u\|_{\bH^{\gamma}_p(t)} +C\|f\|_{\bH^{\gamma}_p(t)}
  \\
  &&+ C\|g\|_{\bH^{\gamma+c_0}_p(t,l_2)}+C\|h\|_{\bH^{\gamma+\bar{c}_0}_p(t,l_2,d_1)}.
   \end{eqnarray*}
   Hence,
\begin{eqnarray} 
\nonumber
    \|u\|^p_{\bH^{\gamma+2}_p(t)} &\leq& C \|u\|^p_{\bH^{\gamma}_p(t)}+C\|f\|^{p}_{\bH^{\gamma}_p(t)}
  \\
   && +C\|g\|^{p}_{\bH^{\gamma+c_0}_p(t,l_2)}+C\|h\|^{p}_{\bH^{\gamma+\bar{c}_0}_p(t,l_2,d_1)}. \label{eqn 5.21.7}
   \end{eqnarray}
   Combining this, \eqref{eqn 5.21.4} and \eqref{eqn 07.26.3}, we get for each $t\leq T$
   
   \begin{eqnarray}
   \nonumber
   \|u\|^p_{\bH^{\gamma}_p(t)}&\leq& C \int^t_0 (t-s)^{\theta-1}\|u\|^p_{\bH^{\gamma}_p(s)}ds +C\|f\|_{\bH^{\gamma}_p(T)}  \\
   && +C\|g\|_{\bH^{\gamma+c_0}_p(T,l_2)}+C\|h\|_{\bH^{\gamma+\bar{c}_0}_p(T,l_2,d_1)}.     \label{final}
   \end{eqnarray}
   We use \eqref{final} and Gronwall's inequality (see \cite {ye2007generalized}) to estimate $\|u\|^p_{\bH^{\gamma}_p(T)}$. Then, applying this estimate to  \eqref{eqn 5.21.7} and  using \eqref{eqn 5.21.4},  we  get a priori estimate  \eqref{eqn 10.22.14:25} if $u_0=v_0=0$.

   Step 2. We consider non-zero initial condition. Let $v\in \cH^{\gamma+2}_p(T)$ denote the solution to equation \eqref{eqn 07.18.3} taken from Theorem \ref{thm 07.24.1}.
    Then $\bar{u}:=u-v\in \cH^{\gamma+2}_p(T)$ satisfies   equation  \eqref{eqn 07.16.2} with $u_0=v_0=0$, $\tilde{f}, \tilde{g}$ and $\tilde{h}$, where
    $$
    \tilde{f}:= (a^{ij}-\delta^{ij})v_{x^ix^j}+b^iv_{x^i}+cv, 
     \quad \tilde{g}^k:=\mu^{ik}v_{x^i}+\nu^kv, \quad \tilde{h}^{rk}:=\bar{\mu}^{rk}v_{x^i}+\bar{\nu}^{rk}v.
    $$ 
    By the result of Step 1, 
    \begin{eqnarray}
    \nonumber
    \|u-v\|_{\cH^{\gamma+2}_p(T)} &\leq& C \|\tilde{f}\|_{\bH^{\gamma}_{p}(T)}+C\|\tilde{g}\|_{\bH^{\gamma+c_{0}}_{p}(T,l_{2})}+C
    \|\tilde{h}\|_{\bH^{\gamma+\bar{c}_{0}}_{p}(T,l_{2},d_{1})}\\
    &\leq& C \|v\|_{\bH^{\gamma+2}_p(T)}. \label{final ineq}
    \end{eqnarray}   
   For the second inequality above we used the calculations in Step 1 (see \eqref{eqn 5.21.4}). 
   Combining \eqref{final ineq} with the estimate for $v$, that is \eqref{eqn 07.19.1}, we finally get  a priori estimate \eqref{eqn 10.22.14:25} for $u$.    Hence,
   the theorem for the linear case is proved.

\textbf{2. Non-linear case}.
First, set
$$
\bar{f}=b^{i}u_{x^{i}}+cu+f(u), \quad \bar{g}^{k}=\mu^{ik}u_{x^{i}}+\nu^{k}u+g^{k}(u),\quad \bar{h}^{rk}=\bar{\mu}^{irk}u_{x^{i}}+\bar{\nu}^{rk}u+h^{rk}(u).
$$
Then by Assumption \ref{asm 07.16.1}, \eqref{eqn 10.08.13:18}, and Assumption \ref{asm 10.08.10:29} we have
\begin{equation*}
\begin{aligned}
&\|\bar{f}(u)-\bar{f}(v)\|_{H^{\gamma}_{p}}+\|\bar{g}(u)-\bar{g}(v)\|_{H^{\gamma+c_{0}}_{p}(l_{2})}+\|\bar{h}(u)-\bar{h}(v)\|_{H^{\gamma+\bar{c}_{0}}_{p}(l_{2},d_{1})}
\\
& \leq C \left( \|u-v\|_{H^{\gamma+1}_{p}}+  1_{c_{0}<1}  \|\mu^{i}(u-v)_{x^{i}}\|_{H^{\gamma+c_{0}}_{p}(l_{2})}+\|u-v\|_{H^{\gamma+c_{0}}_{p}(l_{2})}  \right)
\\
&\quad + C \left(1_{\bar{c}_{0}<1} \|\bar{\mu}^{i}(u-v)_{x^{i}}\|_{H^{\gamma+\bar{c}_{0}}_{p}(l_{2},d_{1})}+\|u-v\|_{H^{\gamma+\bar{c}_{0}}_{p}(l_{2},d_{1})} \right)
\\
& \quad +\|f(u)-f(v)\|_{H^{\gamma}_{p}}+\|g(u)-g(v)\|_{H^{\gamma+c_{0}}_{p}(l_{2})}+\|h(u)-h(v)\|_{H^{\gamma+\bar{c}_{0}}_{p}(l_{2},d_{1})}
\\
& \leq \varepsilon \|u-v\|_{H^{\gamma+2}_{p}}+C\|u-v\|_{H^{\gamma}_{p}},
\end{aligned}
\end{equation*}
where $u,v\in H^{\gamma+2}_{p}$ and the constant $C$ depends only on $\alpha,\beta_{1},\beta_{2},d,d_{1},\gamma,p,\delta,\kappa$ and $\varepsilon$. Hence by considering $\bar{f},\bar{g}^{k}$ and $\bar{h}^{rk}$ in place of $f,g^{k}$ and $h^{rk}$ respectively, we may assume that $b^{i}=c=\mu^{ik}=\nu^{k}=0$, and $\bar{\mu}^{irk}=\bar{\nu}^{rk}=0$.

By  the result for the linear case, for each $u\in \mathcal{H}^{\gamma+2}_{p}(T)$, one can define $v=\mathcal{R}u\in \mathcal{H}^{\gamma+2}_{p}(T)$  as the solution to  the equation
\begin{equation*}
\begin{aligned}
&\partial^{\alpha}_{t}v=a^{ij}v_{x^{i}x^{j}}+f(u)+\partial^{\beta_{1}}_{t}\int_{0}^{t}g^{k}(u)dW^{k}_{s}+\partial^{\beta_{2}}_{t}\int_{0}^{t}h^{rk}(u)dZ^{rk}_{s},\quad t>0
\\
&\quad v(0)=u_{0},\quad 1_{\alpha>1}\partial_{t}v(0)=1_{\alpha>1}v_{0},
\end{aligned}
\end{equation*}
and for this solution we have
\begin{equation*}\label{eqn 10.22.14:31}
\begin{aligned}
 \|v\|_{\mathcal{H}^{\gamma+2}_{p}(T)}  &\leq C \big(\|u_{0}\|_{U^{\gamma+2}_{p}}+1_{\alpha>1}\|v_{0}\|_{V^{\gamma+2}_{p}}+ \|f(u)\|_{\mathbb{H}^{\gamma}_{p}(T)}
\\
&\quad\quad\quad + \|g(u)\|_{\mathbb{H}^{\gamma+c_{0}}_{p}(T,l_{2})}+\|h(u)\|_{\mathbb{H}^{\gamma+\bar{c}_{0}}_{p}(T,l_{2},d_{1})}  \big).
\end{aligned}
\end{equation*}
By \eqref{eqn 07.26.3} for any $\varepsilon>0$, $t\leq T$, and $n=1,2,\dots$ we have
\begin{equation*}\label{eqn 10.22.13:14}
\begin{aligned}
  \|\mathcal{R}u-\mathcal{R}v\|^{p}_{\mathcal{H}^{\gamma+2}_{p}(t)}    &\leq C \big( \|f(u)-f(v)\|^{p}_{\mathbb{H}^{\gamma}_{p}(t)} + \|g(u)-g(v)\|^{p}_{\mathbb{H}^{\gamma+c_{0}}_{p}(t,l_{2})} 
\\
& \quad\quad\quad+\|h(u)-h(v)\|^{p}_{\mathbb{H}^{\gamma+\bar{c}_{0}}_{p}(t,l_{2},d_{1})}  \big)
\\
&\leq \varepsilon^{p}  \|u-v\|^{p}_{\mathbb{H}^{\gamma+2}_{p}(t)}+C_{0}\|u-v\|^{p}_{\mathbb{H}^{\gamma}_{p}(t)}
\\
&\leq   \varepsilon^{p}\|u-v\|^{p}_{\mathcal{H}^{\gamma+2}_{p}(t)} +C_{0}\int_{0}^{t}(t-s)^{\theta-1}\|u-v\|^{p}_{\mathcal{H}^{\gamma+2}_{p}(s)}ds,  
\end{aligned}
\end{equation*}
where the constant $C_{0}$ depends also on $\varepsilon$. Therefore, by using the identity
$$
\int_{0}^{t}(t-s_{1})^{\theta-1}\int_{0}^{s_{1}}(s_{1}-s_{2})^{\theta-1}\cdots \int_{0}^{s_{n-1}}(s_{n-1}-s_{n})^{\theta-1}ds_{n}\dots ds_{1}= \frac{\Gamma(\theta)^{n}}{\Gamma(n\theta+1)}t^{n\theta},
$$
and repeating above inequality, 
we get
\begin{equation*}\label{eqn 10.22.16:01}
\begin{aligned}
   \|\mathcal{R}^{n}u-\mathcal{R}^{n}v\|^{p}_{\mathcal{H}^{\gamma+2}_{p}(T)}   
&\leq \sum_{k=0}^{n} \binom{n}{k}\varepsilon^{(n-k)p}(T^{\theta}C_{0})^{k}\frac{\Gamma(\theta)^{k}}{\Gamma(k\theta+1)}    \|u-v\|^{p}_{\mathcal{H}^{\gamma+2}_{p}(T)}  
\\
&\leq 2^{n}\varepsilon^{np}\max_{k}\left(  \frac{(\varepsilon^{-1} T^{\theta}C_{0} \Gamma(\theta))^{k}}{\Gamma(k\theta+1)}  \right) \|u-v\|^{p}_{\mathcal{H}^{\gamma+2}_{p}(T)}.  \end{aligned}
\end{equation*}
Now fix $\varepsilon<1/8$, and note that  the above maximum is finite. 
This implies that if $n$ is large enough, then $\mathcal{R}^{n}$ is a contraction on $\mathcal{H}^{\gamma+2}_{p}(T)$. This proves the existence and uniqueness
results, and estimate \eqref{eqn 10.22.14:25} also follows.
The theorem is proved.

\mysection{Application to L\'evy space-time white noise}

The L\'evy space-time white noise considered in this article is a  generalization of classical space-time white noise in the sense described below. 

Let $H=L_2(\bR^d)$ be a Hilbert space with the inner product $<f,g>=\int_{\bR^d}fgdx$ and orthogonal basis $\{\eta_k:k=1,2,\cdots\}$. An $H$-valued process $\cZ(t)$ is called a cylindrical  Wiener process  if $<\cZ(t),\eta_k>$ $k=1,2,\cdots$ are independent one-dimentional   Wiener  processes. We get (at least formally) 
$$
\cZ(t)=\sum_k \eta_k Z^k_t, \quad \text{where}\,\,  Z^k_t:=<Z(t),\eta_k>.
$$
If $\cZ(t)$ is a cylindrical Wiener process, then 
we call  $\dot{\cZ}(t)$ a space-time white noise. It is known that 
\begin{equation}\label{eqn 01.03.17:24}
\int_{0}^{t} \int_{\bR^{d}} X(s,x) \dot{\cZ}(dxds) = \sum_{k=1}^{\infty} \int_{0}^{t} \int_{\bR^{d}} X(s,x) \eta^{k}(x) dZ^{k}_{s}, \quad \forall\, t \,(a.s)
\end{equation}
for any $X=\zeta(x)1_{(\tau,\sigma]}(t)$, where $\tau,\sigma$ are bounded stopping times and $\zeta\in \Ccinf(\bR^{d})$ (see \cite{Dalang,kry99analytic}). Here, we use Walsh's stochastic integral for space-time white noise.  

Following the definition and relations explained above, we consider a cylindrical L\'evy process $\cZ_t$ such that  $Z^k_t=<\cZ_t,\eta_k>$ are independent one-dimentional L\'evy processes, and we call $\dot{\cZ}_t$ a L\'evy space-time white noise. The integral against $\dot{\cZ}_t$ is defined as \eqref{eqn 01.03.17:24}. 

Actually, the decomposition of general cylindrical L\'evy process into one-dimensional L\'evy process $Z^{k}_{t}$ only guarantees that $Z^{k}_{t}$ are uncorrelated (see  e.g. \cite{cylindrical}).

In this section, we consider the SPDE
\begin{equation}\label{eqn 10.26.14:22}
\begin{aligned}
&\partial^{\alpha}_{t}u=a^{ij}u_{x^{i}x^{j}}+b^{i}u_{x^{i}}+cu+f(u)+\partial^{\beta_{2}}_{t}\int_{0}^{t}h(u)d\mathcal{Z}_{t},
\\
& \quad u(0,\cdot)=u_{0},\quad 1_{\alpha>1}\partial_{t}u(0,\cdot)=1_{\alpha>1}v_{0}
\end{aligned}
\end{equation}
where $a^{ij},b^{i},c$ are functions of $(\omega,t,x)$, and $f$ and $h$ depend on $(\omega,t,x)$ and the unknown $u$. Using the expansion of $\mathcal{Z}_{t}$, we can rewrite \eqref{eqn 10.26.14:22} as
\begin{equation*}
\begin{aligned}
&\partial^{\alpha}_{t}u=a^{ij}u_{x^{i}x^{j}}+b^{i}u_{x^{i}}+cu+f(u)+\partial^{\beta_{2}}_{t}\sum_{k=1}^{\infty}\int_{0}^{t}g^{k}(u)dZ^{k}_{t},
\\
& \quad u(0,\cdot)=u_{0},\quad 1_{\alpha>1}\partial_{t}u(0,\cdot)=1_{\alpha>1}v_{0}
\end{aligned}
\end{equation*}
where $g^{k}(\omega,t,x,u)=h(\omega,t,x,u)\eta^{k}(x)$.

The following result is from \cite[Lemma 7.1]{kim16timefractionalspde}. 
\begin{lemma}\label{lem 10.16.15:41}
Assume
$$
\kappa_{0}\in \left(\frac{d}{2},d \right],\quad 2\leq 2r\leq p, \quad 2r< \frac{d}{d-\kappa_{0}}.
$$
Also assume that $h(x,u)$ is a function of $(x,u)$, and there is a function $\xi=\xi(x)$ such that
$$
|h(x,u)-h(x,v)|\leq \xi(x)|u-v|.
$$
Then for $u,v\in L_{p}$, we have
$$
\|g(u)-g(v)\|_{H^{-\kappa_{0}}_{p}(l_{2})} \leq C \|\xi\|_{L_{2s}} \|u-v\|_{L_{p}},
$$
where $s=r/r-1$, and $C=C(r)<\infty$. In particular, if $r=1$, and $\xi\in L_{\infty}$, then
$$
\|g(u)-g(v)\|_{H^{-\kappa_{0}}_{p}(l_{2})} \leq C  \|u-v\|_{L_{p}}.
$$
\end{lemma}

In this section, we assume that 
\begin{equation}\label{eqn 10.17.13:27}
\beta_{2} < \frac{3}{4}\alpha +\frac{1}{p},
\end{equation}
and the spatial dimension $d$ satisfies
\begin{equation}\label{eqn 10.17.13:27-2}
d<4-\frac{2(2\beta_{2}-2/p)^{+}}{\alpha}=:d_{0}.
\end{equation}
Note that $d_{0}\in(1,4]$, and if $\beta_{2}<\alpha/4+1/p$, then one can take $d=1,2,3$. Also if $\alpha=\beta_{2}=1$ (in this case $p<4$), then $d<4/p\leq 2$, and thus $d$ must be $1$.

\begin{assumption}\label{asm 10.16.15:18}
(i) The coefficients $a^{ij},b^{i}$, and $c$ are $\mathcal{P}\otimes\cB(\bR^{d})$-measurable.

(ii) The functions $f(t,x,u)$ and $h(t,x,u)$ are $\cP\otimes\cB(\R^{d+1})$-measurable.

(iii) For each $\omega,t,x,u$ and $v$,
$$
|f(t,x,u)-f(t,x,v)| \leq K |u-v|,\quad |h(t,x,u)-h(t,x,v)| \leq K \xi(x) |u-v|,
$$
where $\xi$ is a function of $(\omega,t,x)$.
\end{assumption}

\begin{theorem}\label{thm 10.27:15:35}
Let Assumptions \ref{asm 07.10.1} and \ref{asm 10.16.15:18} hold, and
$$
\|f(0)\|_{\mathbb{H}^{-\kappa_{0}-\bar{c}'_{0}}_{p}(T)}+\|h(0)\|_{\mathbb{L}_{p}(T)}+\sup_{\omega,t}\|\xi\|_{L_{2s}}\leq K < \infty,
$$
where $\kappa_{0}$ and $s$ satisfy
\begin{equation}\label{eqn 10.17.13:28}
\frac{d}{2}<\kappa_{0}<\left(2-\frac{(2\beta_{2}-2/p)^{+}}{\alpha}\right)\wedge d, \quad \frac{d}{2\kappa_{0}-d}<s.
\end{equation}
Also assume that the coefficients $a^{ij},b^{i}$ and $c$ satisfy Assumption \ref{asm 07.16.1} with  $\gamma=-\kappa_{0}-\bar{c}_{0}$, $u_{0}\in U^{-\kappa_{0}-\bar{c}_{0}+2}_{p}$, and $v_{0}\in V^{-\kappa_{0}-\bar{c}_{0}+2}_{p}$.    Then equation \eqref{eqn 10.26.14:22} has unique solution $u\in\mathcal{H}^{2-\kappa_{0}-\bar{c}_{0}}_{p}(T)$, and for this solution we have
\begin{equation}\label{eqn 11.09.19:58}
\begin{aligned}
\|u\|_{\mathcal{H}^{2-\kappa_{0}-\bar{c}_{0}}_{p}(T)} \leq C \big( &\|u_{0}\|_{U^{-\kappa_{0}-\bar{c}_{0}+2}_{p}}+1_{\alpha>1}\|v_{0}\|_{V^{-\kappa_{0}-\bar{c}_{0}+2}_{p}}
\\
& + \|f(0)\|_{\mathbb{H}^{-\kappa_{0}-\bar{c}_{0}}_{p}(T)}+\|h(0)\|_{\mathbb{L}_{p}(T)} \big).
\end{aligned}
\end{equation}
\end{theorem}

\begin{proof}
It suffices to check the conditions for Theorem \ref{thm 10.08.10:57} holds for $\gamma=-\kappa_{0}-\bar{c}_{0}$. Since $f(u)$ is Lipschitz continuous, we only need to check the conditions for $g^{k}(u)=\eta^{k}h(u)$. Let $r=s/(s-1)$. Then $2r<d/(d-\kappa_{0})$ due to the assumption on $s$. Since $\gamma+\bar{c}_{0}=-\kappa_{0}$, by Lemma \ref{lem 10.16.15:41} for any $\varepsilon>0$, we have
$$
\|g(u)-g(v)\|_{H^{\gamma+\bar{c}_{0}}_{p}(l_{2})} \leq C\|\xi\|_{L_{2s}}\|u-v\|_{L_{p}} \leq \varepsilon\|u-v\|_{H^{\gamma+2}_{p}}+C(\varepsilon)\|u-v\|_{H^{\gamma}_{p}},
$$
where the second inequality holds due to the assumption on $\kappa_{0}$. 
Therefore, the condition for $g^{k}$ is also fulfilled. Hence, by Theorem \ref{thm 10.08.10:57} we prove the claims of the theorem with estimate \eqref{eqn 10.17.13:28} replaced by
\begin{equation*}
\begin{aligned}
\|u\|_{\mathcal{H}^{2-\kappa_{0}-\bar{c}_{0}}_{p}(T)} \leq C \big( &\|u_{0}\|_{U^{-\kappa_{0}-\bar{c}_{0}+2}_{p}}+1_{\alpha>1}\|v_{0}\|_{V^{-\kappa_{0}-\bar{c}_{0}+2}_{p}}
\\
& + \|f(0)\|_{\mathbb{H}^{-\kappa_{0}-\bar{c}_{0}}_{p}(T)}+\|g(0)\|_{\mathbb{H}^{-\kappa_{0}}_{p}(T,l_{2})} \big).
\end{aligned}
\end{equation*}
Furthermore, by inspecting the proof of Lemma \ref{lem 10.16.15:41}, one can easily check 
$$\|g(0)\|_{\mathbb{H}^{-\kappa_{0}}_{p}(T,l_{2})}\leq C \|h(0)\|_{\mathbb{L}_{p}(T)}.
$$
 Hence, we have \eqref{eqn 10.17.13:28}, and the theorem is proved.
\end{proof}

\begin{remark}
(i) By \eqref{eqn 10.17.13:27-2} one can always choose $\kappa_{0}$ satisfying \eqref{eqn 10.17.13:28}.

(ii)  Note that the  constant $2-\kappa_{0}-\bar{c}'_{0}$ represents the regularity (or differentiability) of the solution with respect to the space variables. By using the definition of $\bar{c}'_{0}$ we have
\begin{equation*}
\begin{aligned}
0<2-\kappa_{0}-\bar{c}_{0}<
\begin{cases}
2-\frac{d}{2}-\frac{2\beta_{2}-2/p}{\alpha} & \beta_{2}>1/p
\\
2-\frac{d}{2} & \beta_{2}\leq 1/p.
\end{cases}
\end{aligned}
\end{equation*}
If $\xi$ is bounded, then one can choose $r=1$. Thus by taking $\kappa_{0}$ sufficiently close to $d/2$, one can make $2-\kappa_{0}-\bar{c}'_{0}$ as close to the above upper bounds as one wishes.
\end{remark}

\begin{remark}
As mentioned in Remark \ref{remark men},
Assumption \ref{asm 07.10.1} (ii) can be dropped in Theorem \ref{thm 10.27:15:35}.
\end{remark}

\vspace{3mm}

\textbf{Acknowledgement}

The authors are sincerely grateful to the anonymous referee for careful reading and many valuable comments.


\begin{thebibliography}{10}
\bibitem{cylindrical}
D. Applebaum, M. Riedle,
\newblock{Cylindrical L\'evy processes in Banach spaces},
\newblock{\em Proc. London Math. Soc},
\textbf{101} (2010), no.3, 697--726.


\bibitem{baleanu12fractional}
D. Baleanu,
\newblock {\em Fractional calculus: models and numerical methods},
\newblock World Scientific, 2012.


\bibitem{bergh2012interpolation}
J. Bergh, J. L{\"o}fstr{\"o}m,
\newblock {\em Interpolation spaces: an introduction},
\newblock Springer Science \& Business Media, 2012.


\bibitem{chen2014lp}
Z.Q. Chen, K. Kim,
\newblock An $L_{p}$-theory for non-divergence form SPDEs driven by L{\'e}vy
  processes,
\newblock {\em Forum Math}, \textbf{26} (2014), 1381--1411.

\bibitem{chen2015fractional}
Z.Q. Chen, K. Kim, P. Kim,
\newblock Fractional time stochastic partial differential equations,
\newblock {\em Stoch. Proc. Appl}, \textbf{125} (2015), no.4, 1470--1499.

\bibitem{Dalang}  
R. Dalang, L. Quer-Sardanyons: 
\newblock{Stochastic integrals for SPDE's: a comparison},
\newblock{\em Expo. Math.} \textbf{29} (2011), no.1, 67--109.


\bibitem{desch2013maximal}
{W. Desch, S.-O. Londen,}
\newblock Maximal regularity for stochastic integral equations,
\newblock {\em J. Appl. Anal}, \textbf{19} (2013), no.1, 125--140.

\bibitem{desch2008stochastic}
{W. Desch, S.-O. Londen,}
\newblock On a stochastic parabolic integral equation,
\newblock {\em Funct. Anal. Evol. Equ}, (2008) 157--169.

\bibitem{desch2011p}
{W. Desch, S.-O. Londen,}
\newblock {An $L_p$-theory for stochastic integral equations},
\newblock {\em J. Evol. Equ}, \textbf{11} (2011), no.2, 287--317.


\bibitem{dong2019lp}
H. Dong, D. Kim,
\newblock $L_{p}$-estimates for time fractional parabolic equations with
  coefficients measurable in time,
\newblock {\em Adv. Math}, \textbf{345} (2019), 289--345.


\bibitem{gorenflo2014mittag}
R. Gorenflo, A.A. Kilbas,
\newblock {\em Mittag-Leffler functions, related topics and applications},
\newblock Springer, 2014.





\bibitem{kim16timefractionalspde}
I. Kim, K. Kim, S. Lim,
\newblock A Sobolev space theory for stochastic partial differential equations
  with time-fractional derivatives,
\newblock {\em Ann. Probab}, \textbf{47} (2019), no.4, 2087--2139.

\bibitem{kim17timefractionalpde}
I. Kim, K. Kim, S. Lim,
\newblock An $L_q(L_p)$-theory for the time fractional evolution equations
  with variable coefficients,
\newblock {\em Adv. Math}, \textbf{306} (2017), 123--176.

\bibitem{kim2014sobolev}
K. Kim,
\newblock A Sobolev space theory for parabolic stochastic PDEs driven by
  L{\'e}vy processes on $C^{1}$-domains,
\newblock {\em Stoch. Proc. Appl}, \textbf{124} (2014), no.1, 440--474.

\bibitem{kim2012lp}
K. Kim, P. Kim,
\newblock An $L_{p}$-theory of a class of stochastic equations with the random
  fractional Laplacian driven by L{\'e}vy processes,
\newblock {\em Stoch. Proc. Appl},
  \textbf{122} (2012), no.12, 3921--3952.


\bibitem{kry99analytic}
N.V. Krylov,
\newblock An analytic approach to SPDEs,
\newblock {\em Stochastic Partial Differential Equations: Six Perspectives, Mathematical Surveys and Monographs}. \textbf{64} (1999), 185--242.



\bibitem{roeckner}
W. Liu, M. R\"ockner, J. L. da Silva,
\newblock {Quasi-linear stochastic partial differential equations with time-fractional derivatives},
\newblock  {\em SIAM J. Math. Anal}, \textbf{50} (2018), no.3, 2588--2607.


\bibitem{mikulevicius2012l_p}
R. Mikulevicius and H. Pragarauskas,
\newblock On $L_{p}$-estimates of some singular integrals related to jump
  processes,
\newblock {\em SIAM J. Math. Anal}, \textbf{44} (2012), no.4, 2305--2328.


\bibitem{podlubny98fractional}
I. Podlubny,
\newblock {\em Fractional differential equations: an introduction to fractional
  derivatives, fractional differential equations, to methods of their solution
  and some of their applications},
\newblock Elsevier, 1998.


\bibitem{protter2005stochastic}
P.E. Protter,
\newblock {\em Stochastic integration and differential equations},
\newblock  Springer, 2005.

\bibitem{richard14fractional}
H. Richard,
\newblock {\em Fractional calculus: an introduction for physicists},
\newblock World Scientific, 2014.

\bibitem{samko93fractional}
S.G. Samko,  A.A. Kilbas,  O.I. Marichev,
\newblock {\em Fractional integrals and derivatives: theory and applications},
\newblock Gordon and Breach, 1993.

\bibitem{triebel1995interpolation}
H. Triebel,
\newblock {\em Interpolation theory, function spaces, differential operators}, 
\newblock North-Holland Publication, 1995.

\bibitem{ye2007generalized}
H. Ye, J. Gao, Y. Ding,
\newblock A generalized Gronwall inequality and its application to a fractional
  differential equation,
\newblock {\em J. Math. Anal. Appl},
  \textbf{328} (2007), no.2, 1075--1081.

\bibitem{zacher2006quasilinear}
R. Zacher,
\newblock Quasilinear parabolic integro-differential equations with nonlinear
  boundary conditions,
\newblock {\em Differ. Integr. Equ}, \textbf{19} (2006), no.10, 1129--1156.

\end{thebibliography}

\end{document}